\documentclass{amsart}
\usepackage{graphicx,amsmath,mathtools,amssymb,epsfig,color}
\usepackage{hyperref,epstopdf,wrapfig,marvosym}
\newtheorem{theorem}{Theorem}[section]

\newtheorem{definition}[theorem]{Definition}
\newtheorem{example}[theorem]{Example}
\newtheorem{proposition}[theorem]{Proposition}
\newtheorem{remark}[theorem]{Remark}

\newtheorem{conjecture}[theorem]{Conjecture}
\numberwithin{equation}{section}
\def \hfillx {\hspace*{-\textwidth} \hfill}

\hypersetup{pdfstartview={XYZ null null 1.00}}
\begin{document}

\title{A Homology Theory For A Special Family Of Semi-groups}

\author{Sujoy Mukherjee}
\thanks{The author was supported by the Presidential Merit Fellowship of the George Washington University.}
\address{Department of Mathematics, The George Washington University, Washington DC, USA.}
\email{sujoymukherjee@gwu.edu}


\date{September 19, 2015, and in revised form, November 16, 2016.}


\begin{abstract}
In this paper, we construct a new homology theory for semi-groups satisfying the self distributivity axiom or the idempotency axiom. Next, we consider the geometric realization corresponding to the homology theory. We continue with the comparison of this homology theory with one term and two term (rack) homology theories of self-distributive algebraic structures. Finally, we propose connections between the homology theory and knot theory via Temperley-Lieb algebras.
\end{abstract}

\maketitle

\tableofcontents

Self-distributive algebraic structures such as quandles and racks are motivated by knot theory. Rack homology(also known as two term homology) for shelves was introduced by Roger A. Fenn, Colin P. Rourke and Brian J. Sanderson \cite{FRS1,FRS2,FRS3}. This was modified into quandle homology by J. Scott Carter, Daniel Jelsovsky, Seiichi Kamada, Laurel Langford and Masahico Saito to define co-cycle invariants for knots \cite{CJKLS}. Later, one term homology for shelves was introduced by J\'ozef H. Przytycki \cite{Prz1}. While quandles are very useful from a knot theoretic point of view, for algebraic purposes one may juggle with the axioms of a quandle to obtain other algebraic structures and study these under the same settings. This note grew out of a similar attempt. While trying to understand the behavior of rack and one term homology of associative shelves I decided to manipulate some axioms and explore. What I obtained is described in this note. J\'ozef H. Przytycki suggested that the homology theory should be called {\bf elbow homology} (or {\bf lbo homology}). To understand this terminology please stare long enough at the graphical interpretation of the face maps in Figure \ref{graphical interpretation of the face maps}! The note is organized as follows.

The following section introduces the necessary tools required for defining the homology theory after which it discusses the main aspects of lbo homology. In Section 2 the geometric aspects of lbo homology have been introduced. Lbo homology has connections to Temperley-Lieb algebras and Jones' monoids. This has been discussed in Section 3. As associative shelves lie in the intersection of associative algebraic structures and self-distributive algebraic structures, there are several homology theories that work for associative shelves. A comparative study has been done in Section 4 along with tables of computations done for lbo homology. Based on these, open questions are discussed. 

\section{Introduction}

The necessary prerequisites for defining the homology theory is introduced in the following subsection.  

\subsection{Preliminaries}

A {\bf shelf} or a right self-distributive algebraic structure is a magma $(X,*)$ satisfying the following axiom for all $a,b,c \in X$:
\begin{equation*}
  (a*b)*c = (a*c)*(b*c).
\end{equation*}

A shelf is called a {\bf rack} if there exists $\bar{*}:X \times X \longrightarrow X$ such that for all $a,b \in X$,
\begin{equation*}
(a\bar{*} b)*b = a = (a*b) \bar{*} b.
\end{equation*}

In addition, if every element in a rack is idempotent, then the algebraic structure obtained is called a {\bf quandle}. On the other hand, if every element in a shelf is idempotent, the resulting algebraic structure is called a {\bf spindle}. The three axioms of a quandle correspond to the three Reidemeister moves and therefore quandles are very useful tool to build invariants of links. However, to work with framed links it is enough to consider racks.

As will be observed later, the associativity axiom is needed for lbo homology. However, the trivial quandle is the unique rack\footnote{The trivial quandle is defined in the following way. Let $(X,*)$ be a magma. For all $a,b \in X$ define $a*b = a$. Here, the trivial quandle is referred to as a rack to emphasize that the second and not the first axiom of a quandle is playing a role here. There are several spindles which are associative.} which satisfies the associativity axiom. This is bad news for lbo homology from a knot theoretic point of view.

In 2015, it was proven that shelves with an identity element (unit) are associative. The proof follows from the following two propositions.

\begin{proposition}[Sam C.]\label{proposition 1 Sam C.}
    For a unital shelf $(X,*)$, the following axioms hold:
    \begin{enumerate}
        \item {$a*a = a$, for all $a \in X$. In other words, $(X,*)$ is a spindle.}
        \item {$a*b = b*(a*b)$ for all $a,b \in X$.}
        \item {$a*b = (a*b)*b$ for all $a,b \in X$.}
    \end{enumerate}
\end{proposition}

A shelf satisfying all the three axioms in the above proposition is called a {\bf pre unital shelf} and a shelf satisfying the second and the third axioms of the above proposition is called a {\bf proto unital shelf}. The converse of Proposition \ref{proposition 1 Sam C.} does not hold but it is easy to observe that removing the unit element from a unital shelf gives a pre unital shelf and adding unit element to a pre unital shelf gives a unital shelf \cite{CMP}.

\begin{proposition}[Sam C.]\label{proposition 2 Sam C.}
    Proto unital shelves are associative.
\end{proposition}

The relations between the various algebraic structures discussed above is summarized in Figure \ref{firstfig}.

\begin{figure}[ht]
%\blankbox{.6\columnwidth}{5pc}
\includegraphics[scale=0.25]{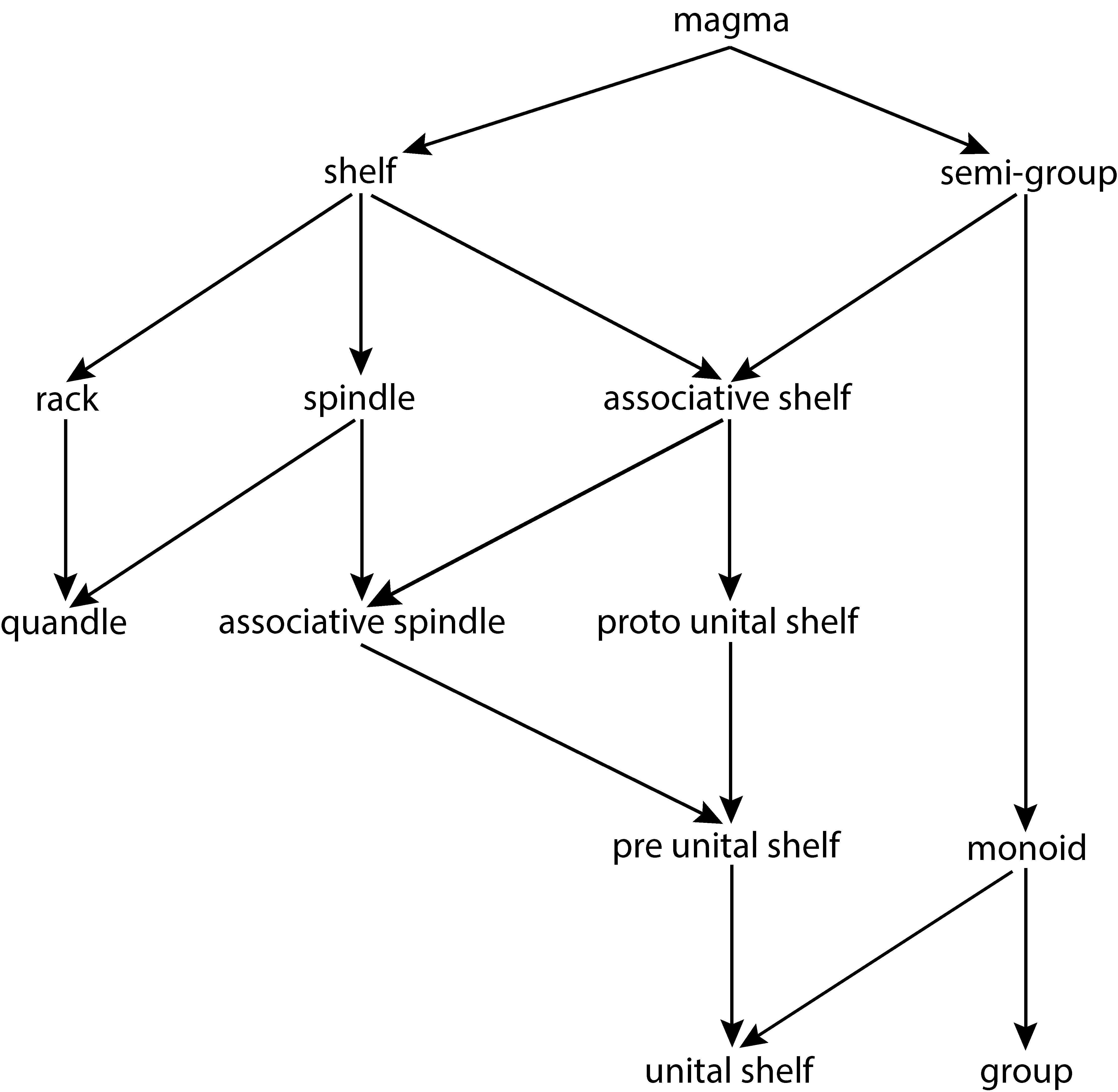}
\caption{Relations between the algebraic structures.}
\label{firstfig}
\end{figure}

\subsection{Lbo homology}

It is well known that to construct a homology theory it is enough to construct a pre-simplicial module (also known as semi-simplicial module). The pre-simplicial module allows for a chain complex to be constructed by defining the boundary map as the alternating sum of the face maps of the pre-simplicial module \cite{Lod}. Lbo homology is developed in the same way.

\begin{definition}
    A {\bf pre-simplicial module} $(C_n,d_{i,n})$ consists of a sequence of $R$-modules $C_n$ over a ring $R$ for $n \geq 0$, and face maps $d_{i,n}:C_n \longrightarrow C_{n-1}$ for all $0 \leq i \leq n$ such that for $i<j$,
    \begin{equation*}
        d_{i,n} \circ d_{j,n+1} = d_{j-1,n} \circ d_{i,n+1}.\footnote{As is usually done, to simplify notation the second index of the face maps will be omitted in the rest of this note.}
    \end{equation*}
\end{definition}

Now let $\partial_n:C_n \longrightarrow C_{n-1}$ and for all $x\in C_n$, let
\begin{equation*}
    \partial_n(x) = \sum_{i=0}^n (-1)^i d_i(x).
\end{equation*}

It is not difficult to observe that $\partial_n \circ \partial_{n+1} = 0$, so that $(C_n,\partial_n)$ is a chain complex. Following is the same construction in the context of lbo homology.

Let $(X,*)$ be an associative shelf and $C_n=\mathbb{Z}X^{n+1}$ for $n \geq 0$ and trivial otherwise. Let $d_i:C_n \longrightarrow C_{n-1}$ for $0 \leq i \leq n$ be given by,

\begin{equation}\label{1.1}
d_{i}(x_0,x_1,\cdots,x_n)=
\begin{cases*}
(x_0*x_1,x_2,\cdots,x_{n-1},x_n*x_0) & if i=0,\\
(x_n*x_0,x_1,\cdots,x_{n-2},x_{n-1}*x_n) & if i=n,\\
(x_0,x_1,\cdots,x_{i-2},x_{i-1}*x_i,x_i*x_{i+1},x_{i+2},\cdots,x_n) & else.
\end{cases*}
\end{equation}

\begin{figure}[ht]
	%blankbox{.6\columnwidth}{5pc}
	\includegraphics[scale=0.25]{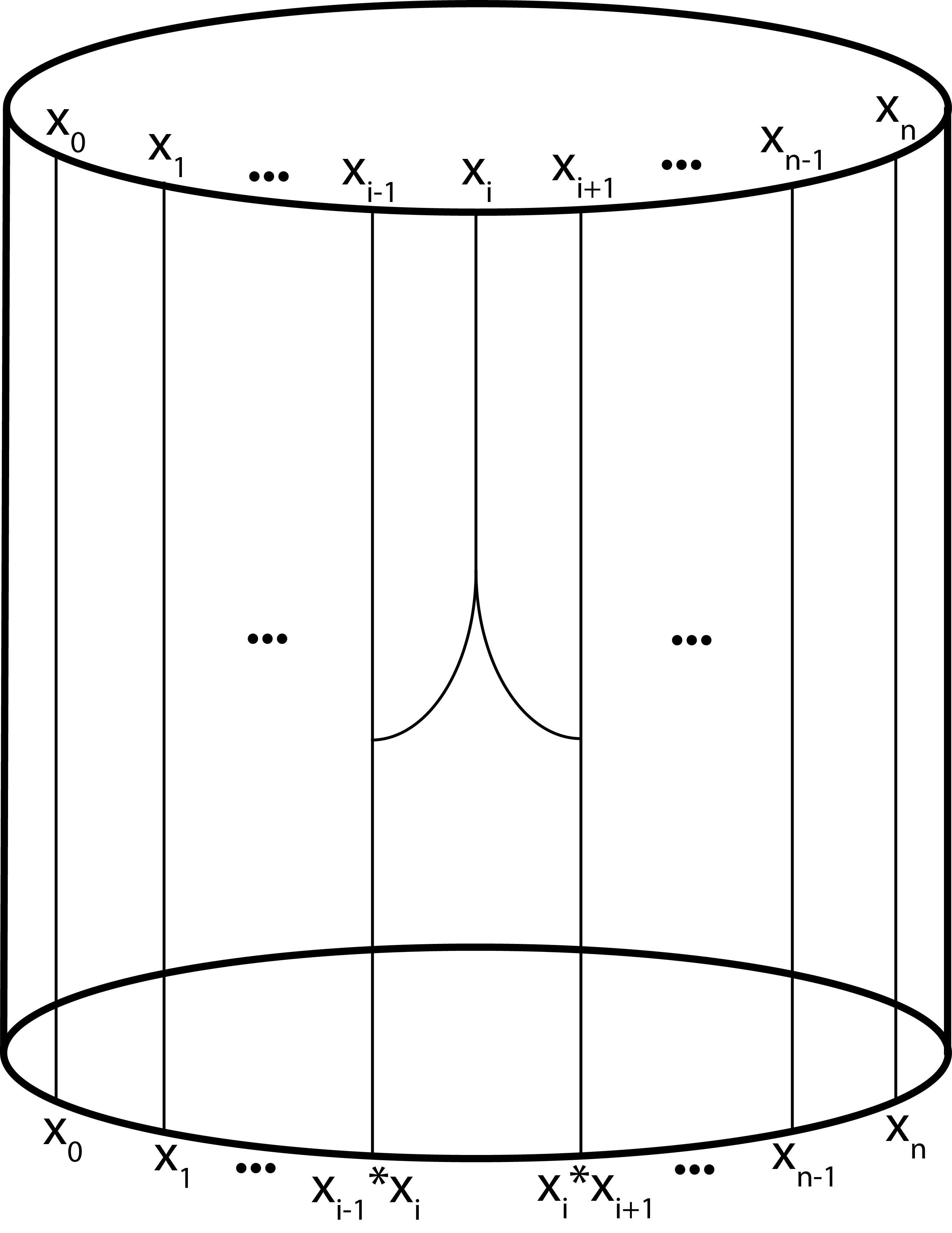}
	\caption{Graphical interpretation of the $i^{th}$ face map of lbo homology.}
	\label{graphical interpretation of the face maps}
\end{figure}

The next thing to prove is that $(C_n,d_i)$ is a pre-simplicial module. But before proving this, some examples are considered to illustrate how the face maps use the axioms of an associative shelf. A non-trivial case in the above definition would occur when $n=1$. Part (1) of Example \ref{examples of computation of face maps} explains that case. 

\begin{example}\label{examples of computation of face maps}
Let $(x_0,x_1,...,x_n) \in \mathbb{Z}X^{n+1}$.
\begin{enumerate}
    \item{For $n=1$, $d_0(x_0,x_1) = x_0*x_1*x_0$, and $d_1(x_0,x_1) = x_1*x_0*x_1$. Here, multiplication on the left and then on the right is the same as multiplication on the right and then on the left as $(X,*)$ is a semi-group.}

    \item{Let $n=2,i=0,j=2$. Then $d_0 \circ d_2 (x_0,x_1,x_2) = d_0(x_2*x_0,x_1*x_2) = (x_2*x_0)*(x_1*x_2)*(x_2*x_0).$ $d_1 \circ d_0(x_0,x_1,x_2) = d_1(x_0*x_1,x_2*x_0) = (x_2*x_0)*(x_0*x_1)*(x_2*x_0)$.}
    
    \item{Let $n=3,i=1,j=2$. Then $d_1 \circ d_2 (x_0,x_1,x_2,x_3) = d_1 (x_0, x_1*x_2,x_2*x_3) = (x_0*(x_1*x_2),(x_1*x_2)*(x_2*x_3))$. $d_1 \circ d_1 (x_0,x_1,x_2,x_3) = d_1(x_0*x_1,x_1*x_2,x_3) = ((x_0*x_1)*(x_1*x_2),(x_1*x_2)*x_3).$}

\end{enumerate}

\end{example}

The example above implies that the axioms necessary for $(C_n,d_i)$ to be a pre-simplicial module are the associativity axiom and $a*b*b*c=a*b*c$ for $a,b,c \in X$. The second axiom holds for associative shelves as,
\begin{equation}\label{lengthaxiom equation}
    a*b*b*c=a*b*c*b*c=a*c*b*c=a*b*c.
\end{equation}

\begin{proposition}\label{existential proof of lbo homology}
    $(C_n,d_i)$ is a pre-simplicial module.
\end{proposition}

\begin{proof}
    The proof of this proposition requires a careful organization of the various possibilities for the sole axiom of a pre-simplicial module and dividing them in to multiple cases.
    \begin{enumerate}
        \item[Case 1:]{For $n=0,1,2,3$, checking each of the small number of possibilities suffices.}
        \item[Case 2:]{Let $n>3,0<i<j<n,j=i+1,$ that is when $i$ and $j$ are adjacent to each other. An arbitrary element in the domain of $d_i$ looks like $(...,x_{i-1},x_i,x_{i+1},x_{i+2},...)$. Then,
        \begin{equation*}
        \begin{split}
         & d_i \circ d_j (...,x_{i-1},x_i,x_{i+1},x_{i+2},...) \\  = \ & d_i \circ d_{i+1} (...,x_{i-1},x_i,x_{i+1},x_{i+2},...)\\ = \ & d_i(...,x_{i-1},x_i*x_{i+1},x_{i+1}*x_{i+2},...)\\ = \ & (...,x_{i-1}*(x_i*x_{i+1}),(x_i*x_{i+1})*(x_{i+1}*x_{i+2}),...). 
        \end{split}    
        \end{equation*}
        
        On the other hand,
        \begin{equation*}
            \begin{split}
                & d_{j-1} \circ d_i (...,x_{i-1},x_i,x_{i+1},x_{i+2},...) \\ = \ & d_i \circ d_i (...,x_{i-1},x_i,x_{i+1},x_{i+2},...)\\ = \ &  d_i(...,x_{i-1}*x_i,x_i*x_{i+1},x_{i+2},...)\\ = \ & (...,(x_{i-1}*x_i)*(x_i*x_{i+1}),(x_i*x_{i+1})*x_{i+2},...).      
            \end{split}
        \end{equation*}
           
        Under the axioms of associative shelves and using \ref{lengthaxiom equation} it follows that both sides are equal.}
        \item[Case 3:]{Let $n>3,0<i<j<n,j=i+2.$ Now $i$ and $j$ are neighbors but not adjacent. Let $(...,x_{i-1},x_i,x_{i+1},x_{i+2},x_{i+3},...)$ be an arbitrary element in the domain of $d_i$. Then,
        \begin{equation*}
            \begin{split}
                & d_i \circ d_j (...,x_{i-1},x_i,x_{i+1},x_{i+2},x_{i+3},...) \\ = \ & d_i \circ d_{i+2} (...,x_{i-1},x_i,x_{i+1},x_{i+2},x_{i+3},...)\\ = \ & d_i(...,x_{i-1},x_i,x_{i+1}*x_{i+2},x_{i+2}*x_{i+3},...)\\ = \ & (...,x_{i-1}*x_i,x_i*(x_{i+1}*x_{i+2}),x_{i+2}*x_{i+3},...).      
            \end{split}
        \end{equation*}
        On the other hand,
        \begin{equation*}
            \begin{split}
               & d_i \circ d_j (...,x_{i-1},x_i,x_{i+1},x_{i+2},x_{i+3},...) \\ = \ & d_{i+1} \circ d_i (...,x_{i-1},x_i,x_{i+1},x_{i+2},x_{i+3},...)\\ = \ & d_{i+1}(...,x_{i-1}*x_i,x_i*x_{i+1},x_{i+2},x_{i+3},...)\\ = \ & (...,x_{i-1}*x_i,(x_i*x_{i+1})*x_{i+2},x_{i+2}*x_{i+3},...).  
            \end{split}
        \end{equation*}}
        
        In this case, associativity is the only axiom necessary to show the equality of the two expressions.
        \item[Case 4:]{Let $n>3,0<i<j<n,j-i>2$. This case is almost trivial as the $d_i$ face map and the $d_j$ face map do not interact with each other at all and therefore the axiom of a pre-simplicial module follows immediately without the need of any axioms of an associative shelf.}
    \end{enumerate}
    
    The cases when $j = n$, that is when $j$ is at the right end is similar except a shift in indexes and are left for the reader to verify.
    
    \iffalse
        In each of the above cases except the first, $n-(j-i)>2$. As the face maps of the {\it potential} pre-simplicial module are cyclic in nature, this rules out the possibility of the $d_j$ face map affecting the $d_i$ map by coming around, or vice-versa. If otherwise, it is dealt with in a separate case and is not mentioned here to avoid increasing the length of this proof.
	\fi

\end{proof}

The above proposition and defining the boundary maps $\partial_n$ in the usual way as the alternate sum of the face maps of the pre-simplicial module leads on to the following theorem.

\begin{theorem}
    $(C_n,\partial_n)$ is a chain complex.
\end{theorem}

The homology groups obtained from the chain complex in the above theorem will be called the {\bf lbo homology groups}. The boundary maps of this homology are cyclic in nature. It turns out that by using similar face maps but not in a cyclic manner, it is possible to form a different homology theory. This homology theory is described briefly in the following subsection. 

It is important to note here that while the associativity axiom is always necessary for lbo homology, it is possible to exchange the right self-distributivity axiom with idempotence. Therefore, lbo homology also works if the algebraic structure of interest is an idempotent semi-group. Although there is a non-trivial intersection between associative shelves and idempotent semi-groups in the form of associative spindles, there is no containment in either direction. See Figure \ref{Relations between associative shelves and idempotent semi-groups}. Further, one can also work with semi-groups satisfying the axiom $a*b*b*c = a*b*c$ for all $a,b,c \in$ semi-group.

\begin{figure}[ht]
%\blankbox{.6\columnwidth}{5pc}
\includegraphics[scale=0.25]{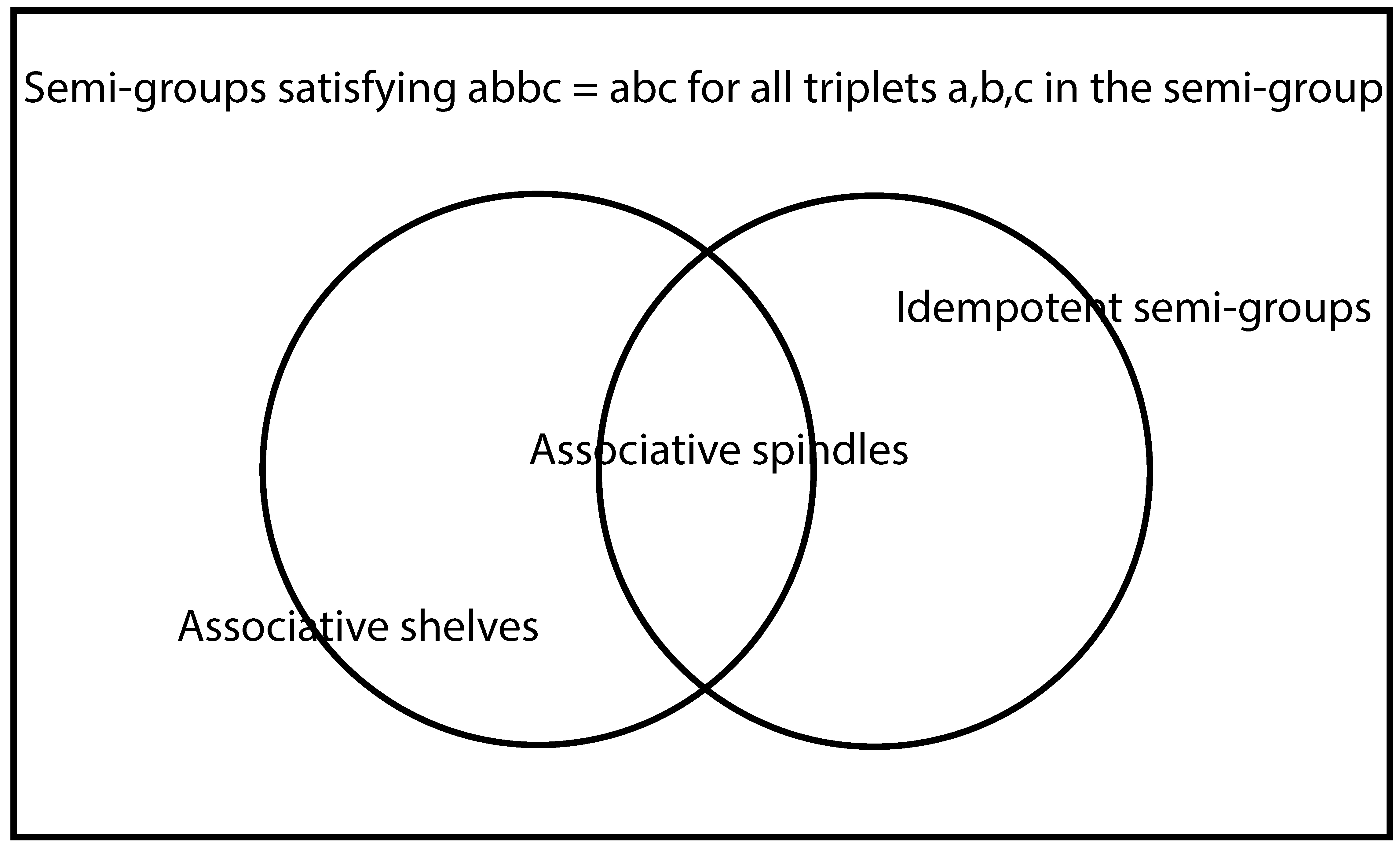}
\caption{Relations between associative shelves and idempotent semi-groups.}
\label{Relations between associative shelves and idempotent semi-groups}
\end{figure}

Example \ref{examples of computation of face maps} part (1), illustrates the face maps $d_0$ and $d_1$ when $n=1$. It follows that $\partial_1(x_0,x_1) = x_0*x_1*x_0 - x_1*x_0*x_1$, for $(x_0,x_1) \in \mathbb{Z}X^2$. This implies that lbo homology measures {\it how far} a proto unital shelf or idempotent semi-group is from being commutative. In particular, when the algebraic structure under consideration is finite and commutative, the zero homology group is $\mathbb{Z}^{|X|}$. 

It would be useful to extend the pre-simplicial module leading on to lbo homology to a simplicial module for the useful implications. I do not know how to do it. The pre-simplicial module leading on to lbo homology, however, can be extended to a {\it very weak simplicial module} for special families of associative shelves and idempotent semi-groups. The notion of a very weak simplicial module was introduced in \cite{Prz1,Prz2}.

\begin{definition}
    A {\bf simplicial module} $(C_n,d_{i,n},s_{j,n})$ consists of a sequence of $R$-modules $C_n$ over a ring $R$ for $n \geq 0$, face maps $d_{i,n}:C_n \longrightarrow C_{n-1}$, and degeneracy maps $s_{j,n}:C_n \longrightarrow C_{n+1}$ for all $0 \leq i,j \leq n$ such that,
    \begin{enumerate}
        \item{$d_{i,n} \circ d_{j,n+1} = d_{j-1,n} \circ d_{i,n+1}, \ \ for \ i<j.$}
        \item{$s_{i,n} \circ s_{j,n-1} = s_{j+1,n} \circ s_{i,n-1}, \ \ for \ i \leq j.$}
        \item{$d_{i,n} \circ s_{j,n-1} =
        \begin{cases*}
            s_{j-1,n-1} \circ d_{i,n}, & for $i<j$,\\
            s_{j,n-1} \circ d_{i-1,n}, & for $i > j+1$.
        \end{cases*}$}
        \item{$d_{i,n+1} \circ s_{j,n} = d_{i+1,n+1} \circ s_{j,n} = Id_{C_n}.$\footnote{As for pre-simplicial modules, to simplify notation the second index of the face maps will be omitted in the rest of this note.}}    
    \end{enumerate}
\end{definition}

A {\bf very weak simplicial module} need not satisfy axiom 4 in the above definition. Now consider the pre-simplicial module $(C_n,d_{i,n})$ leading on to lbo homology. To convert it to a very weak simplicial module a degeneracy map is necessary which satisfies axioms 2 and 3 in the last definition.

\begin{definition}
	Let $(X,*)$ be a magma. If there exists an element $e \in X$ such that $x * e = e = e * x$ for all $x\in X$ then $e$ is called a {\bf zero}. 
\end{definition}

There are many associative shelves and idempotent semi-groups with zeros. Table \ref{IS and AS} shows one example of each. In particular, any associative shelf or idempotent semi-group can be extended to another associative shelf or idempotent semi-group by adding a zero element. This follows as both the associative axiom and the self-distributive axiom work for triplets involving the zero element.

\begin{table}[h]
\caption{An idempotent semi-group and an associative shelf.}
\label{IS and AS}
	\begin{minipage}{0.5\textwidth}
		\centering
		\begin{tabular}{c | c c c c}
			$*$&0&1&2&3 \\
			\hline
			0&0&0&0&0\\
			1&0&1&1&3\\
			2&0&1&2&3\\
			3&0&1&1&3\\	
		\end{tabular}
	\end{minipage}
	\hfillx
	\begin{minipage}{0.5\textwidth}
		\centering
		\begin{tabular}{c|c c c c}
			$*$&0&1&2&3\\
			\hline
			0&0&0&0&0\\
			1&0&0&1&1\\
			2&0&0&2&2\\
			3&0&0&2&3\\
		\end{tabular}
	\end{minipage}
\end{table}

Let $(X,*)$ be an associative shelf or an idempotent semi-group with zero element $e$. Consider homomorphisms (degeneracy maps) $s_{j,n}:C_n \longrightarrow C_{n+1}$ for all $0 \leq j \leq n$ given by
\begin{equation*}
    s_{j,n}(x_0,x_1,...,x_n) = (x_0,x_1,...,x_{j-1},e,e,x_{j+1},...,x_n).
\end{equation*}

\begin{proposition}
    $(C_n,d_i,s_j)$, the pre-simplicial module leading on to lbo homology along with the $s_j$ maps defined above is a very weak simplicial module.
\end{proposition}

\begin{proof}
    To prove the proposition, it is necessary to verify the three axioms of a very weak simplicial module. The first axiom follows from Proposition \ref{existential proof of lbo homology}. For the second and the third axioms, the various possibilities are tested.
    
    \begin{itemize}

    \item[A2:]{$s_i \circ s_j = s_{j+1} \circ s_i,$ for $i \leq j.$\
    \begin{enumerate}
        \item[Case 1:]{As before, for $n = 0,1,2,3$, checking all the possibilities one by one suffices.}
        
        \item[Case 2:]{Let $ n > 3,0<i=j<n.$ Let $(...,x_{i-1},x_i,x_{i+1},x_{i+2},...)$ be an arbitrary element in the domain of $s_j.$ Then,
        \begin{equation*}
        \begin{split}
         & s_i \circ s_j (...,x_{i-1},x_i,x_{i+1},x_{i+2},...) \\  = \ & s_i \circ s_i (...,x_{i-1},x_i,x_{i+1},x_{i+2},...)\\ = \ & s_i(...,x_{i-1},e,e,x_{i+1},x_{i+2},...)\\ = \ & (...,x_{i-1},e,e,e,x_{i+1},x_{i+2},...). 
        \end{split}    
        \end{equation*}
        On the other hand,
        \begin{equation*}
        \begin{split}
         & s_{j+1} \circ s_i (...,x_{i-1},x_i,x_{i+1},x_{i+2},...) \\  = \ & s_{i+1} \circ s_i (...,x_{i-1},x_i,x_{i+1},x_{i+2},...)\\ = \ & s_{i+1}(...,x_{i-1},e,e,x_{i+1},x_{i+2},...)\\ = \ & (...,x_{i-1},e,e,e,x_{i+1},x_{i+2},...). 
        \end{split}    
        \end{equation*}
        Therefore, both the sides are equal. Note that no additional axioms were needed.}
        
        \item[Case 3:]{Let $ n > 3,0<i<j<n,j = i+1.$ Let $(...,x_{i-1},x_i,x_{i+1},x_{i+2},...)$ be an arbitrary element in the domain of $s_j.$ Then,
        \begin{equation*}
        \begin{split}
         & s_i \circ s_j (...,x_{i-1},x_i,x_{i+1},x_{i+2},...) \\  = \ & s_i \circ s_{i+1} (...,x_{i-1},x_i,x_{i+1},x_{i+2},...)\\ = \ & s_i(...,x_{i-1},x_i,e,e,x_{i+2},...)\\ = \ & (...,x_{i-1},e,e,e,e,x_{i+2},...). 
        \end{split}    
        \end{equation*}
        On the other hand,
        \begin{equation*}
        \begin{split}
         & s_{j+1} \circ s_i (...,x_{i-1},x_i,x_{i+1},x_{i+2},...) \\  = \ & s_{i+2} \circ s_i (...,x_{i-1},x_i,x_{i+1},x_{i+2},...)\\ = \ & s_{i+2}(...,x_{i-1},e,e,x_{i+1},x_{i+2},...)\\ = \ & (...,x_{i-1},e,e,e,e,x_{i+2},...). 
        \end{split}    
        \end{equation*}
        Once again, both the sides are equal without the need of any additional axiom.}
        The computations for the cases when $j>i+1$ are similar and equality holds as there is no interaction between the two maps like in the previous case.
    \end{enumerate}
    }

    \item[A3:]{$d_{i,n} \circ s_{j,n-1} =
        \begin{cases*}
            s_{j-1,n-1} \circ d_{i,n}, & for $i<j$,\\
            s_{j,n-1} \circ d_{i-1,n}, & for $i > j+1$.
        \end{cases*}$\
    \begin{enumerate}
        \item[Case 1:]{Once again for $n=0,1,2,3$, the small number of possibilities are verified.}
        
        \item[Case 2:]{Let $n>3,0<i<j<n,j=i+1,$ Let $(...,x_{i-1},x_i,x_{i+1},x_{i+2},...)$ be an arbitrary element in the domain of $d_i,s_j$. Then,
        \begin{equation*}
        \begin{split}
         & d_i \circ s_j (...,x_{i-1},x_i,x_{i+1},x_{i+2},...) \\  = \ & d_i \circ s_{i+1} (...,x_{i-1},x_i,x_{i+1},x_{i+2},...)\\ = \ & d_i(...,x_{i-1},x_i,e,e,x_{i+2},...)\\ = \ & (...,x_{i-1}*x_i,x_i*e,e,x_{i+2},...). 
        \end{split}    
        \end{equation*}
        On the other hand,
        \begin{equation*}
        \begin{split}
         & s_{j-1} \circ d_i (...,x_{i-1},x_i,x_{i+1},x_{i+2},...) \\  = \ & s_i \circ d_i (...,x_{i-1},x_i,x_{i+1},x_{i+2},...)\\ = \ & s_i(...,x_{i-1}*x_i,x_i*x_{i+1},x_{i+2},...)\\ = \ & (...,x_{i-1}*x_i,e,e,x_{i+2},...). 
        \end{split}    
        \end{equation*}
        Therefore, both the sides are equal as $e$ is a right zero.
        }
        
        \item[Case 3:]{Let $n>3,0<i<j<n,j=i+2,$ Let $(...,x_{i-1},x_i,x_{i+1},x_{i+2},x_{i+3},...)$ be an arbitrary element in the domain of $d_i,s_j$. Then,
        \begin{equation*}
        \begin{split}
         & d_i \circ s_j (...,x_{i-1},x_i,x_{i+1},x_{i+2},x_{i+3},...) \\  = \ & d_i \circ s_{i+2} (...,x_{i-1},x_i,x_{i+1},x_{i+2},x_{i+3},...)\\ = \ & d_i(...,x_{i-1},x_i,x_{i+1},e,e,x_{i+3},...)\\ = \ & (...,x_{i-1}*x_i,x_i*x_{i+1},e,e,x_{i+3},...). 
        \end{split}    
        \end{equation*}
        On the other hand,
        \begin{equation*}
        \begin{split}
         & s_{j-1} \circ d_i (...,x_{i-1},x_i,x_{i+1},x_{i+2},x_{i+3},...) \\  = \ & s_{i+1} \circ d_i (...,x_{i-1},x_i,x_{i+1},x_{i+2},x_{i+3},...)\\ = \ & s_{i+1}(...,x_{i-1}*x_i,x_i*x_{i+1},x_{i+2},x_{i+3},...)\\ = \ & (...,x_{i-1}*x_i,x_i*x_{i+1},e,e,x_{i+3},...). 
        \end{split}    
        \end{equation*}
        Equality holds once again without the need of any additional axiom.
        }
        The remaining cases for the first part of the axiom are similar to the last case as there is no interaction between the two maps and therefore are not computed separately.
    \end{enumerate}
    The computations for the second part of the axiom are similar. However, it should be noted that in the first part of the axiom, right zero was required. For the second part of the axiom to hold, left zero is necessary. Moreover, verification of the cases for $j=n$ just require a shift in indexes and are left for the reader to verify.
    }
    \end{itemize}
\end{proof}

It is possible to construct degenerate sub-complexes for very weak simplicial modules by dividing by the relation $d_i \circ s_i - d_{i+1} \circ s_i$ analogous to degenerate sub-complexes for simplicial modules \cite{Prz2}. 

\subsection{A non-cyclic version of lbo homology}

As was mentioned in the last subsection, for the same algebraic structures it is possible to form a different (not completely!) homology theory which is non-cyclic and less interesting. The pre-simplicial module leading on to this version of lbo homology is as follows.

Let $(X,*)$ be a semi-group satisfying $a*b*b*c = a*b*c$ and $C_n^{nc}=\mathbb{Z}X^{n+1}$ for $n \geq 0$ and trivial otherwise. Let $d_i:C_n^{nc} \longrightarrow C_{n-1}^{nc}$ for $0 \leq i \leq n$ be given by,

\begin{equation}\label{1.2}
d_{i}(x_0,x_1,\cdots,x_n)=
\begin{cases*}
(x_0*x_1,x_2,\cdots,x_{n-1},x_n) & if i=0,\\
(x_0,x_1,\cdots,x_{n-2},x_{n-1}*x_n) & if i=n,\\
(x_0,x_1,\cdots,x_{i-2},x_{i-1}*x_i,x_i*x_{i+1},x_{i+2},\cdots,x_n) & else.
\end{cases*}
\end{equation}

\ref{1.1} and \ref{1.2} have different face maps in the extreme cases. It is not very difficult to observe that all the conditions of a pre-simplicial module are satisfied and therefore after defining boundary maps $\partial_n^{nc}$ the theorem below follows.

\begin{theorem}
    $(C_n^{nc},\partial_n^{nc})$ is a chain complex.
\end{theorem}

In particular, $\partial_1(x_0,x_1) = x_0*x_1 - x_0*x_1 = 0$, for $(x_0,x_1) \in \mathbb{Z}X^2$. Therefore, for a semi-group $(X,*)$ satisfying $a*b*b*c = a*b*c$, $H_0^{nc}(X) = C_0^{nc} =  \mathbb{Z}X.$ 
\iffalse
The following proposition shows that the homology groups of idempotent semi-groups are trivial in all positive dimensions.

\begin{proposition}\label{non-cyclic for idempotnent}
	For an idempotent semi-group $(X,*)$, $H_n^{nc}(X) = 0$ for $n > 0.$
\end{proposition}

\begin{proof}
	Let $(C_n^{nc},\partial_n^{nc})$ be the chain complex corresponding to the idempotent semi-group $(X,*)$. To prove the theorem it is enough to construct a chain homotopy between the identity chain map and the zero chain map. Define a sequence of maps $f_n: \mathbb{Z}X^{n} \longrightarrow \mathbb{Z}X^{n+1}$ given by, $$f_n(x_0,x_1,...,x_n) = (-1)^{n+1}(x_0,x_1,...,x_n,x_n).$$
	
	Then,
	\begin{equation*}
		\begin{split}
		&(\partial_{n+1}^{nc} \circ f_n)(x_0,x_1,...,x_n) = (-1)^{n+1}\partial_{n+1}^{nc}(x_0,x_1,...,x_n,x_n)\\ \ &= (-1)^{n+1}\{(x_0*x_1,x_2,...,x_n,x_n) + \sum_{i=1}^{n}(-1)^i(x_0,...,x_{i-1}*x_i,x_i*x_{i+1},...,x_n,x_n)\\ \ &+ (-1)^{n+1}(x_0,x_1,...,x_n)\},
		\end{split}
	\end{equation*} and,
	
	 \begin{equation*}
	 \begin{split}
		 & (f_{n-1} \circ \partial_n^{nc})(x_0,x_1,...x_n) = f_{n-1}\{(x_0*x_1,x_2,...,x_n)\\ \ &+ \sum_{i=1}^{n-1} (-1)^i(x_0,...,x_{i-1}*x_i,x_i*x_{i+1},...,x_n) + (x_0,x_1,...,x_{n-1}*x_n)\}\\ \ &= (-1)^n \{(x_0*x_1,x_2,...,x_n,x_n)\\ \ &+ \sum_{i=1}^{n-1} (-1)^i(x_0,...,x_{i-1}*x_i,x_i*x_{i+1},...,x_n,x_n) + (x_0,x_1,...,x_{n-1}*x_n,x_n)\}.
	 \end{split}
	 \end{equation*}
	 
	 Taking the sum of the two equations, $$(\partial_{n+1}^{nc} \circ f_n)(x_0,x_1,...,x_n) + (f_{n-1} \circ \partial_n^{nc})(x_0,x_1,...,x_n) = (x_0,x_1,...,x_n).$$
\end{proof}

\fi

\section{Geometric realization of the pre-simplicial set leading to lbo homology}

In a pre-simplicial module if the chain modules are replaced with sets (i.e. without the module structure) with appropriate differentials, the resulting structure is known as a {\bf pre-simplicial set} (also known as semi-simplicial complex). A pre-simplicial set leads on to a CW complex which is built by gluing simplexes together using the face maps of the pre-simplicial set \cite{Gie,Hu,Lod,Mil,Prz2}. The advantage of doing this is that the homology groups of a homology theory which is completely developed algebraically can be computed by using  simplicial homology which is very well developed. Moreover, the construction being well defined, the geometric structures obtained at each dimension are invariants of the underlying algebraic structure. The idea, with more details in the context of lbo homology is as follows.

\begin{definition}
    A pre-simplicial set $(X_n,d_{i,n})$ consists of a sequence of sets $X_n$ for $n \geq 0$, and face maps $d_{i,n}:X_n \longrightarrow X_{n-1}$ for all $0 \leq i \leq n$ such that for $i < j$,
    \begin{equation*}
        d_{i,n} \circ d_{j,n+1} = d_{j-1,n} \circ d_{i,n+1}.\footnote{As for pre-simplicial modules, to simplify notation the second index of the face maps will be omitted in the rest of this note.}
    \end{equation*}
\end{definition}

\begin{figure}[ht]
%\blankbox{.6\columnwidth}{5pc}
\includegraphics[scale=0.35]{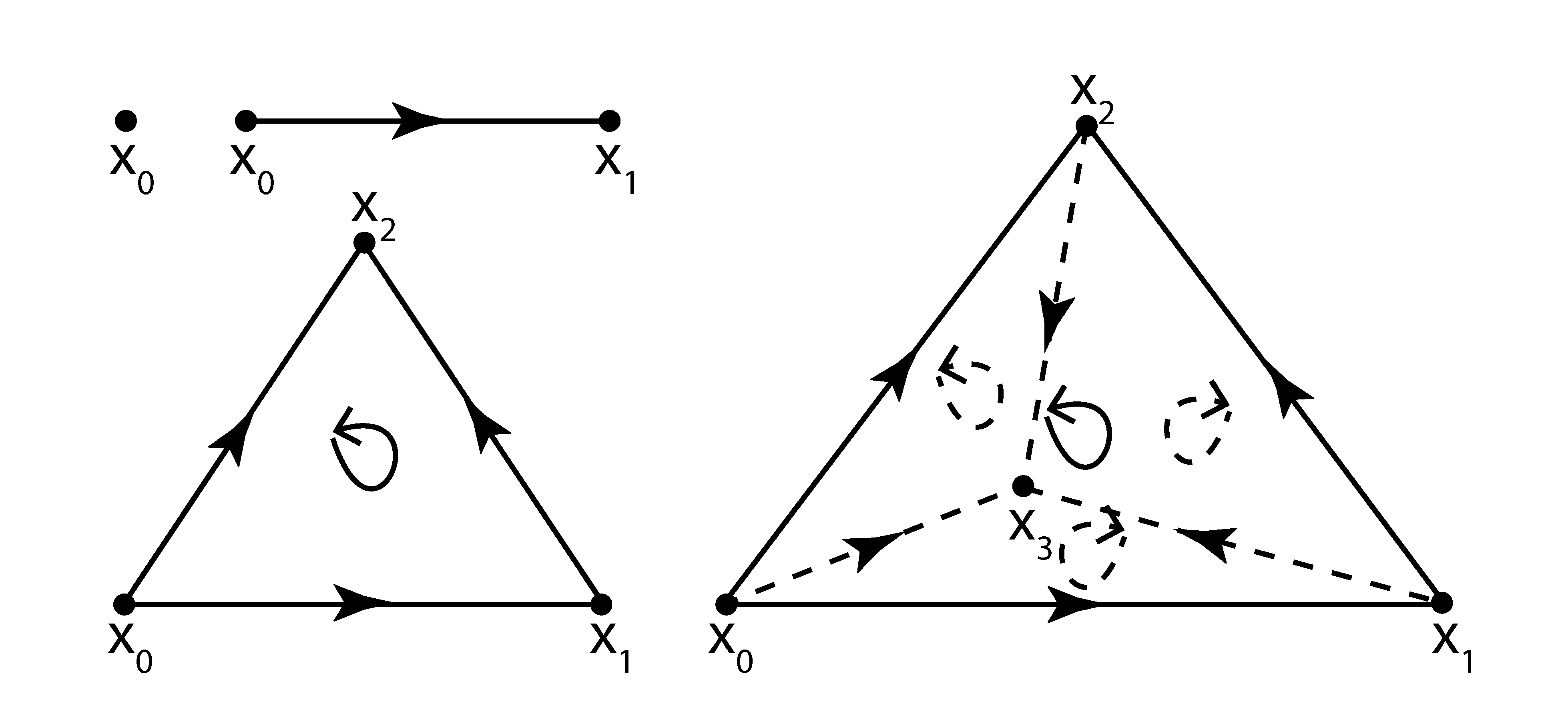}
\caption{The standard $0$, $1$, $2$, and $3$ dimensional simplexes.}
\label{cells of the CW complex}
\end{figure}

A pre-simplicial set leads to a pre-simplicial module by changing the sequence of sets $X_n$ to a sequence of free modules $X_n$ over a ring $R$. The converse is not true in general. Therefore, a pre-simplicial set can be converted in to a homology theory by the technique discussed in the last section. By definition of lbo homology and Proposition \ref{existential proof of lbo homology}, the pre-simplicial module leading on to lbo homology is a pre-simplicial set.

Let $\mathcal{L}$ be a pre-simplicial set. Let the category  $\Delta$ be the co-pre-simplicial space with standard simplices as the objects,
\begin{equation*}
    \Delta^n = \{(y_0,y_1,...,y_n) \in \mathbb{R}^{n+1} \ such \ that \ \sum_{i=0}^n y_i = 1, y_i \geq 0\}.    
\end{equation*}
The morphisms (co-face maps) are defined by $d^i:\Delta^n \longrightarrow \Delta^{n+1}$ given by 
\begin{equation*}
    d^i(y_0,y_1,...,y_n) = (y_0,y_1,...,y_{i-1},0,y_i,..., y_n).
\end{equation*}
The co-face maps satisfy $d^i \circ d^{j-1} = d^j \circ d^i$ for $i<j.$ Then the geometric realization of $\mathcal{L}$ denoted by $|\mathcal{L}|$ is the CW complex associated to $\mathcal{L}$, that is for $x \in X_n$ and $y\in \Delta^{n-1}$,
\begin{equation*}
    |\mathcal{L}| =  \frac{\bigsqcup_{n \geq 0}(X_n \times \Delta^n)}{(x, d^i(y)) = (d_i(x), y)}.  
\end{equation*}
Here $X_n$ has discrete topology and $|\mathcal{L}|$ has quotient topology. Figure \ref{cells of the CW complex} shows the standard $0$, $1$, $2$, and $3$ dimensional cells.

Now let us visualize the first few cells of the CW complex corresponding to lbo homology. For this, let us recollect the first few boundary maps. Let $(x_0,x_1,...x_n) \in X^n.$

\begin{equation*}
    \partial_1(x_0,x_1) = x_0*x_1*x_0 - x_1*x_0*x_1.
\end{equation*}
\begin{equation*}
    \partial_2(x_0,x_1,x_2) = (x_0*x_1,x_2*x_0) - (x_0*x_1,x_1*x_2) + (x_2*x_0,x_1*x_2).
\end{equation*}
\begin{equation*}
\begin{split}
    \partial_3(x_0,x_1,x_2,x_3) & = (x_0*x_1,x_2,x_3*x_0) - (x_0*x_1,x_1*x_2,x_3) \\ & + (x_0,x_1*x_2,x_2*x_3) - (x_3*x_0,x_1,x_2*x_3).
\end{split}
\end{equation*}

\begin{figure}[ht]
	%\blankbox{.6\columnwidth}{5pc}
	\includegraphics[scale=0.5]{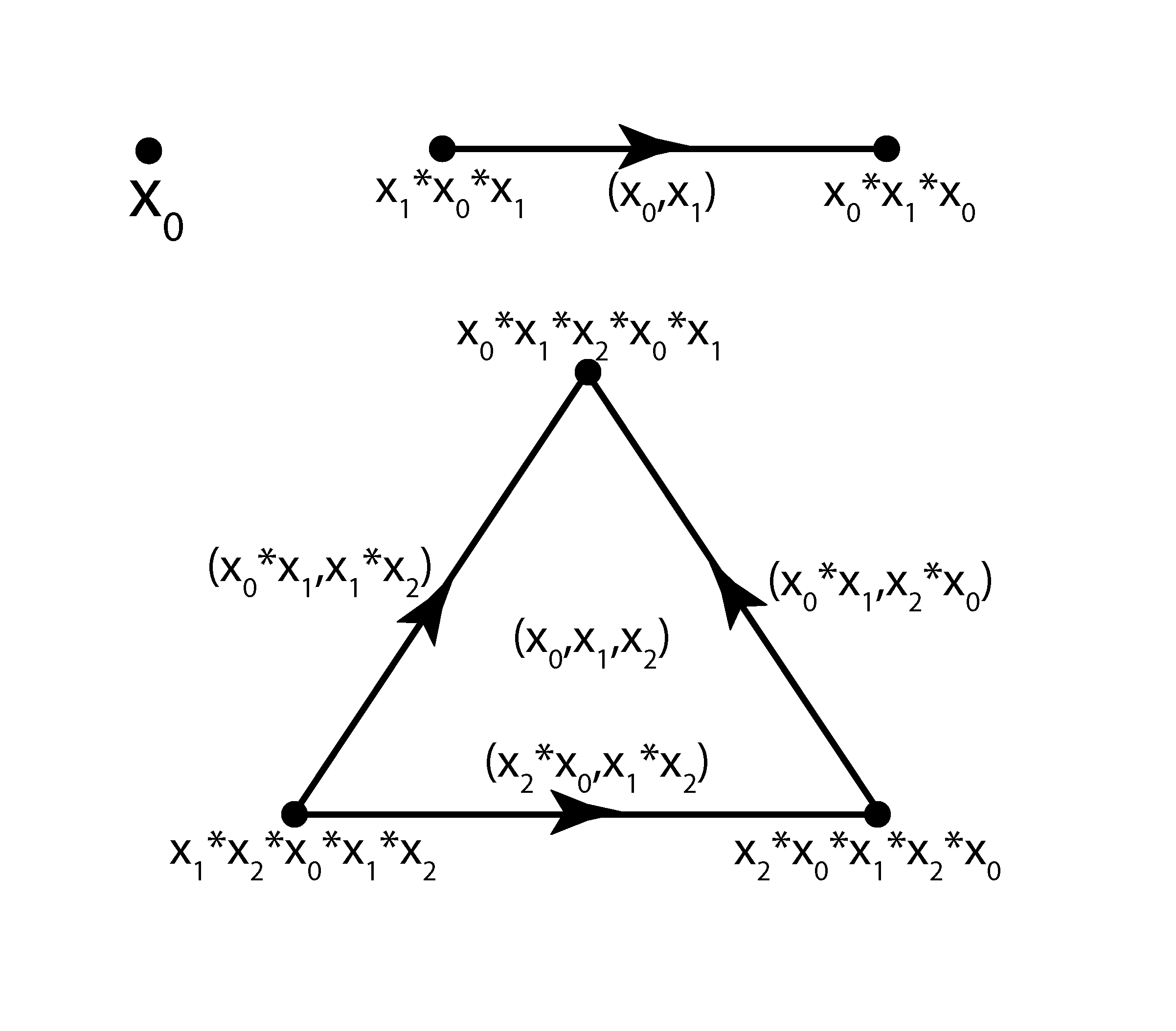}
	\caption{The $0$, $1$, and $2$ dimensional simplexes from lbo homology.}
	\label{cells_from_lbo_homology}
	\end{figure}

Corresponding to lbo homology let $\mathcal{L}$ be the pre-simplicial set defined for a finite associative shelf or a finite idempotent semi-group $(X,*)$ having $n$ elements. Then the $0$-cells of the CW complex are $n$ points, each corresponding to one element from the algebraic structure. The $1$ dimensional cells are edges corresponding to the gluing instructions obtained from the face maps of lbo homology and so on. Figures \ref{cells_from_lbo_homology} and \ref{cell_in_third_dimension_from_lbo_homology} demonstrate the geometric realization of the individual cells obtained from lbo homology in small dimensions. In Figure \ref{cell_in_third_dimension_from_lbo_homology}, $(x_0,x_1,x_2,x_3)$ denotes the $3$ dimensional cell, the blue tuples represent the four triangles, the red ones represent the six edges and the expressions in black correspond to the four vertices.\footnote{For the vertices, the parentheses are not used for simpler notation.} The figure is easily interpreted after the following routine calculation.

\begin{equation*}
\begin{split}
    \partial_3(x_0,x_1,x_2,x_3) &= (x_0*x_1,x_2,x_3*x_0) - (x_0*x_1,x_1*x_2,x_3) \\ & + (x_0,x_1*x_2,x_2*x_3) - (x_3*x_0,x_1,x_2*x_3).
\end{split}
\end{equation*}
\begin{equation*}
\begin{split}
    \partial_2(x_0*x_1,x_2,x_3*x_0) &= (x_0*x_1*x_2,x_3*x_0*x_1) - (x_0*x_1*x_2,x_2*x_3*x_0)\\ &+ (x_3*x_0*x_1,x_2*x_3*x_0).
\end{split}
\end{equation*}
\begin{equation*}
\begin{split}
    \partial_2(x_0*x_1,x_1*x_2,x_3) &= (x_0*x_1*x_2,x_3*x_0*x_1) - (x_0*x_1*x_2,x_1*x_2*x_3)\\ &+ (x_3*x_0*x_1,x_1*x_2*x_3).
\end{split}
\end{equation*}
\begin{equation*}
\begin{split}
    \partial_2(x_0,x_1*x_2,x_2*x_3) &= (x_0*x_1*x_2,x_2*x_3*x_0) - (x_0*x_1*x_2,x_1*x_2*x_3)\\ &+ (x_2*x_3*x_0,x_1*x_2*x_3).
\end{split}
\end{equation*}
\begin{equation*}
\begin{split}
    \partial_2(x_3*x_0,x_1,x_2*x_3) &= (x_3*x_0*x_1,x_2*x_3*x_0) - (x_3*x_0*x_1,x_1*x_2*x_3)\\ &+ (x_2*x_3*x_0,x_1*x_2*x_3).
\end{split}
\end{equation*}
\begin{equation*}
    \partial_1(x_0*x_1*x_2,x_3*x_0*x_1) = (x_0*x_3*x_0*x_1*x_2)-(x_3*x_2*x_3*x_0*x_1).
\end{equation*}
\begin{equation*}
    \partial_1(x_0*x_1*x_2,x_2*x_3*x_0) = (x_0*x_3*x_0*x_1*x_2)-(x_2*x_1*x_2*x_3*x_0).
\end{equation*}
\begin{equation*}
    \partial_1(x_3*x_0*x_1,x_2*x_3*x_0) = (x_3*x_2*x_3*x_0*x_1)-(x_2*x_1*x_2*x_3*x_0).
\end{equation*}
\begin{equation*}
    \partial_1(x_0*x_1*x_2,x_1*x_2*x_3) = (x_0*x_3*x_0*x_1*x_2)-(x_1*x_0*x_1*x_2*x_3).
\end{equation*}
\begin{equation*}
    \partial_1(x_3*x_0*x_1,x_1*x_2*x_3) = (x_3*x_2*x_3*x_0*x_1)-(x_1*x_0*x_1*x_2*x_3).
\end{equation*}
\begin{equation*}
    \partial_1(x_2*x_3*x_0,x_1*x_2*x_3) = (x_2*x_1*x_2*x_3*x_0)-(x_1*x_0*x_1*x_2*x_3).
\end{equation*}

\begin{figure}[ht]
%\blankbox{.6\columnwidth}{5pc}
\includegraphics[scale=0.3]{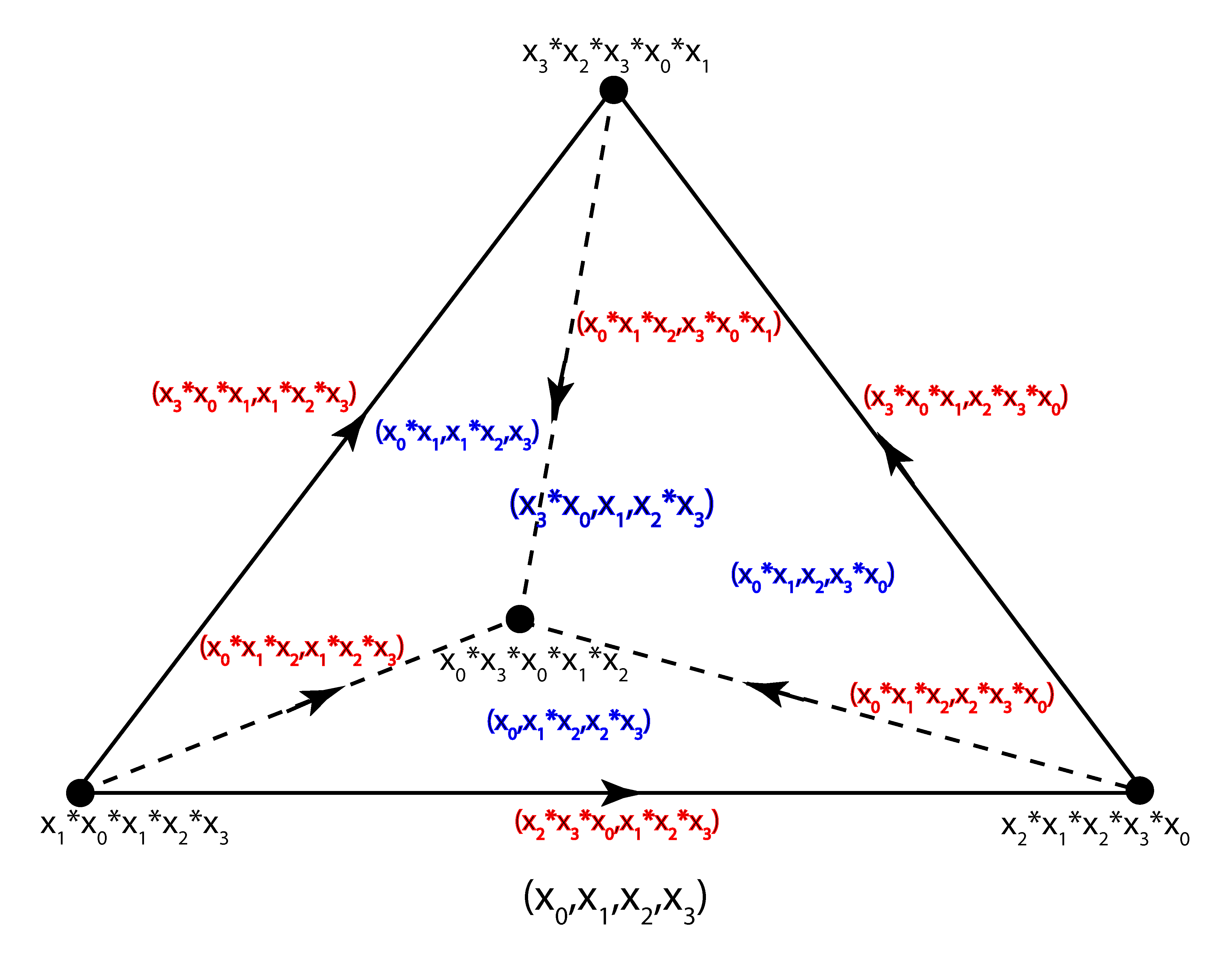}
\caption{The $3$ dimensional simplex from lbo homology.}
\label{cell_in_third_dimension_from_lbo_homology}
\end{figure}

\begin{figure}[ht]
	%\blankbox{.6\columnwidth}{5pc}
	\includegraphics[scale=0.25]{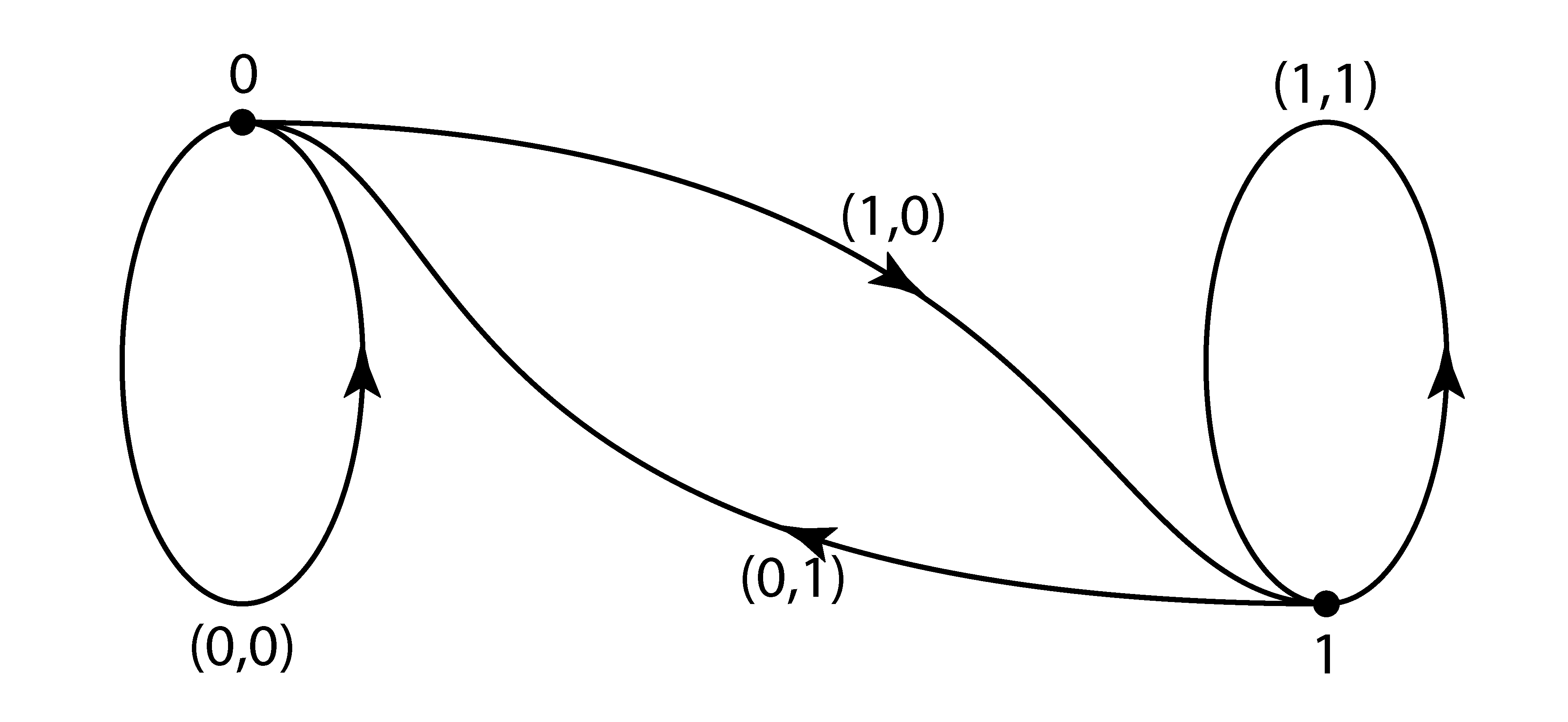}
	\caption{The 1-skeleton of the CW complex.}
	\label{till two dimension}
\end{figure}

Let us now understand the geometric realization for the associative shelf shown in Table \ref{small associative shelf}.  

\begin{wraptable}{r}{0.42 \textwidth}\label{small associative shelf}
\caption{A small associative shelf.}
\centering
\begin{tabular}{c | c c}
		$*$&0&1 \\
		\hline
		0&0&0\\
		1&1&1\\
\end{tabular}
\end{wraptable}

There are two elements in the associative shelf. Therefore, in the geometric realization there are two $0$ dimensional simplexes (vertices) denoted by $0$ and $1$. There are four $1$ dimensional simplexes (edges) and eight $2$ dimensional simplexes (triangles).

$\partial_1(0,0) = 0 - 0,$ and is therefore represented by a loop starting at $0$. Due to same reason, $(1,1)$ is represented by a loop starting at $1$. $\partial_1(0,1) = 0*1*0 - 1*0*1.$ Using the multiplication table of the associative shelf we obtain, $\partial_1(0,1) = 0 - 1.$ Similarly, $\partial_1(1,0) = 1 - 0.$ The geometric realization of the $1$-skeleton is demonstrated in Figure \ref{till two dimension}. The geometric structure is homotopy equivalent to a wedge of three circles $(\mathbb{S}^1 \vee \mathbb{S}^1 \vee \mathbb{S}^1)$ and has $0^{th}$ homology $\mathbb{Z}$ which matches with the $0^{th}$ lbo homology group of the associative shelf. 

Taking into account the second boundary map to visualize the complex up to the third dimension gets a little complicated.

\begin{equation*}
\begin{split}
    & \partial_2(0,0,0) = (0,0) - (0,0) + (0,0). \\
    & \partial_2(0,0,1) = (0,1) - (0,0) + (1,0). \\
    & \partial_2(0,1,0) = (0,0) - (0,1) + (0,1). \\
    & \partial_2(0,1,1) = (0,1) - (0,1) + (1,1). \\
    & \partial_2(1,0,0) = (1,0) - (1,0) + (0,0). \\
    & \partial_2(1,0,1) = (1,1) - (1,0) + (1,0). \\
    & \partial_2(1,1,0) = (1,0) - (1,1) + (0,1). \\
    & \partial_2(1,1,1) = (1,1) - (1,1) + (1,1). \\
\end{split}
\end{equation*}
    
The simplexes $(0,0,0)$ and $(1,1,1)$ fill up the closed curves formed by $(0,0)$ and $(1,1)$ in Figure \ref{till two dimension} and form contractible discs (triangles). The simplexes $(0,0,1)$ and $(1,1,0)$ are symmetric up to change of indices. Figure \ref{hanky} shows the geometric realization of $(0,0,1).$ The simplexes $(0,1,0)$, $(0,1,1)$, $(1,0,0)$, and $(1,0,1)$ are homotopic to cones over the loops $(0,0)$ or $(1,1)$. The resulting structure after gluing all the simplexes appropriately to Figure \ref{till two dimension} looks like a cylinder capped off at both ends with three hollow spaces inside. From the ends of the cylinder with the caps as bases two cones start and are attached to the cylinder along the edges $(0,1)$ and $(1,0)$. In particular, the fundamental group of the geometric realization of $X$ is trivial and therefore $H_1(X) = 0$. In fact, computer calculations show that $H_i(X) = 0$ for $0<i<10$.

\begin{figure}[ht]
	%\blankbox{.6\columnwidth}{5pc}
	\includegraphics[scale=0.25]{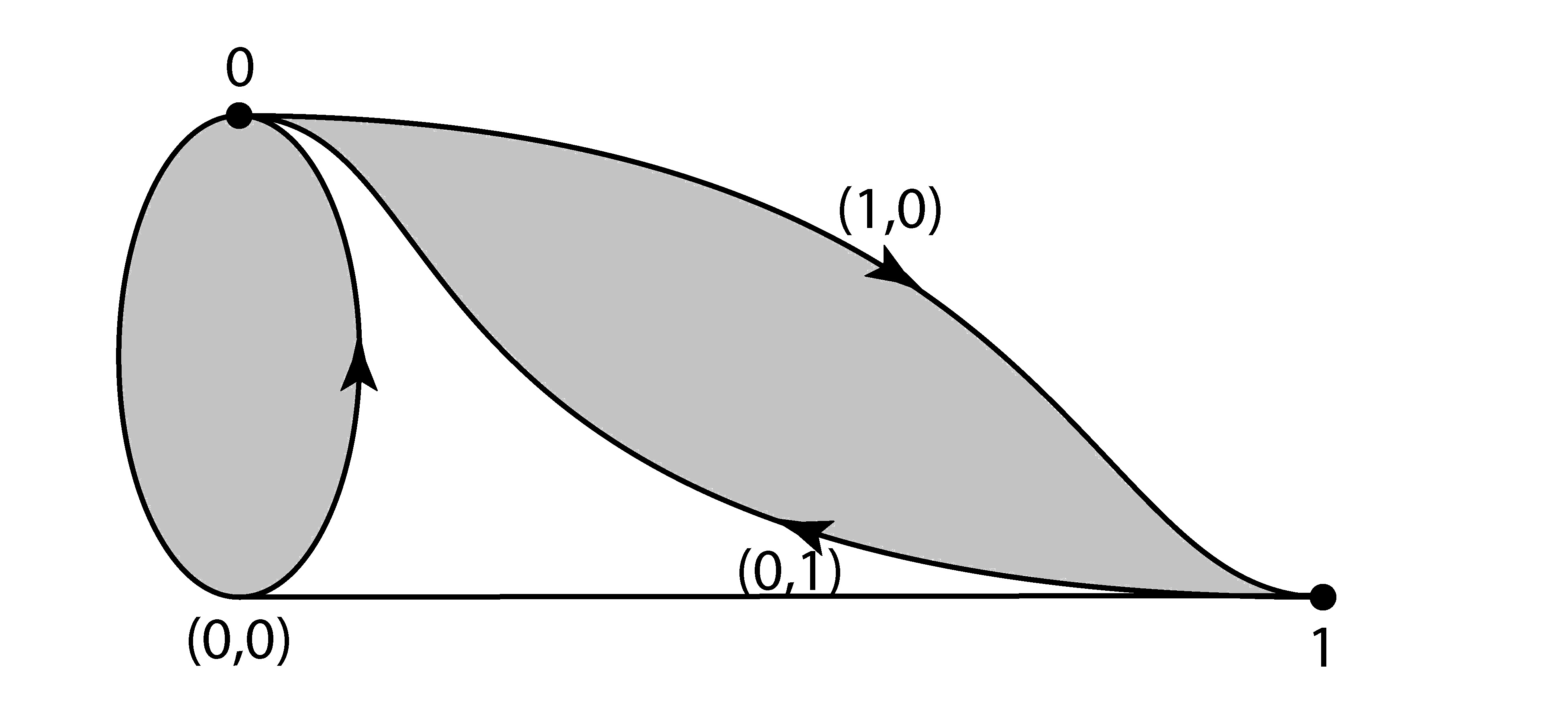}
	\caption{Geometric realization of the simplex $(0,0,1)$.}
	\label{hanky}
\end{figure}

\subsection{A comparison between algebraic and geometric arguments}

This subsection comprises of the interpretation of $H_0(X)$ when $X$ is a proto unital shelf or idempotent semi-group. The interpretation is first done algebraically in the form of Propositions \ref{rem} and \ref{prop 2.2} followed by the interpretation using geometric arguments. In particular, the following discussion provides a simpler computation technique for $H_0(X)$ when $X$ is a finite proto unital shelf or a finite idempotent semi-group. However, before that consider the following observations which prove the equivalence of the commutativity axiom and $\partial_0 = 0$ (braid group relation) in idempotent semi-groups and proto unital shelves. 

\begin{proposition}\label{rem} 
	Let $(X,*)$ be a proto unital shelf or an idempotent semi-group and $a,b \in X$. Then $a*b = b*a$ if and only if $a*b*a = b*a*b.$
\end{proposition}
\begin{proof}
	\
	\begin{enumerate}
		\item{For a proto unital shelf, $ a*b*a = b*a*b \Longleftrightarrow b*a = a*b.$\\
			In other words, $a*b \neq b*a \Longleftrightarrow a*b*a \neq b*a*b.$}
		\item{For an idempotent semi-group $a*b*a = b*a*b \implies a*b*a*b = b*a*b*b \implies a*b = b*a*b.$\\
			Similarly, $a*b*a = b*a*b \implies b*a*b*a = b*b*a*b \implies b*a = b*a*b.$\\
			Together, they imply $a*b = b*a.$\\
			Conversely, $a*b =b*a \implies b*a*b = b*b*a = b*a,$ and $a*b = b*a \implies a*b*a = b*a*a = b*a.$\\
			Together, they imply $a*b*a = b*a*b.$\\
			Therefore, $a*b \neq b*a \Longleftrightarrow a*b*a \neq b*a*b.$}
	\end{enumerate}
\end{proof}

\begin{figure}[ht]
	%\blankbox{.6\columnwidth}{5pc}
	\includegraphics[scale=0.5]{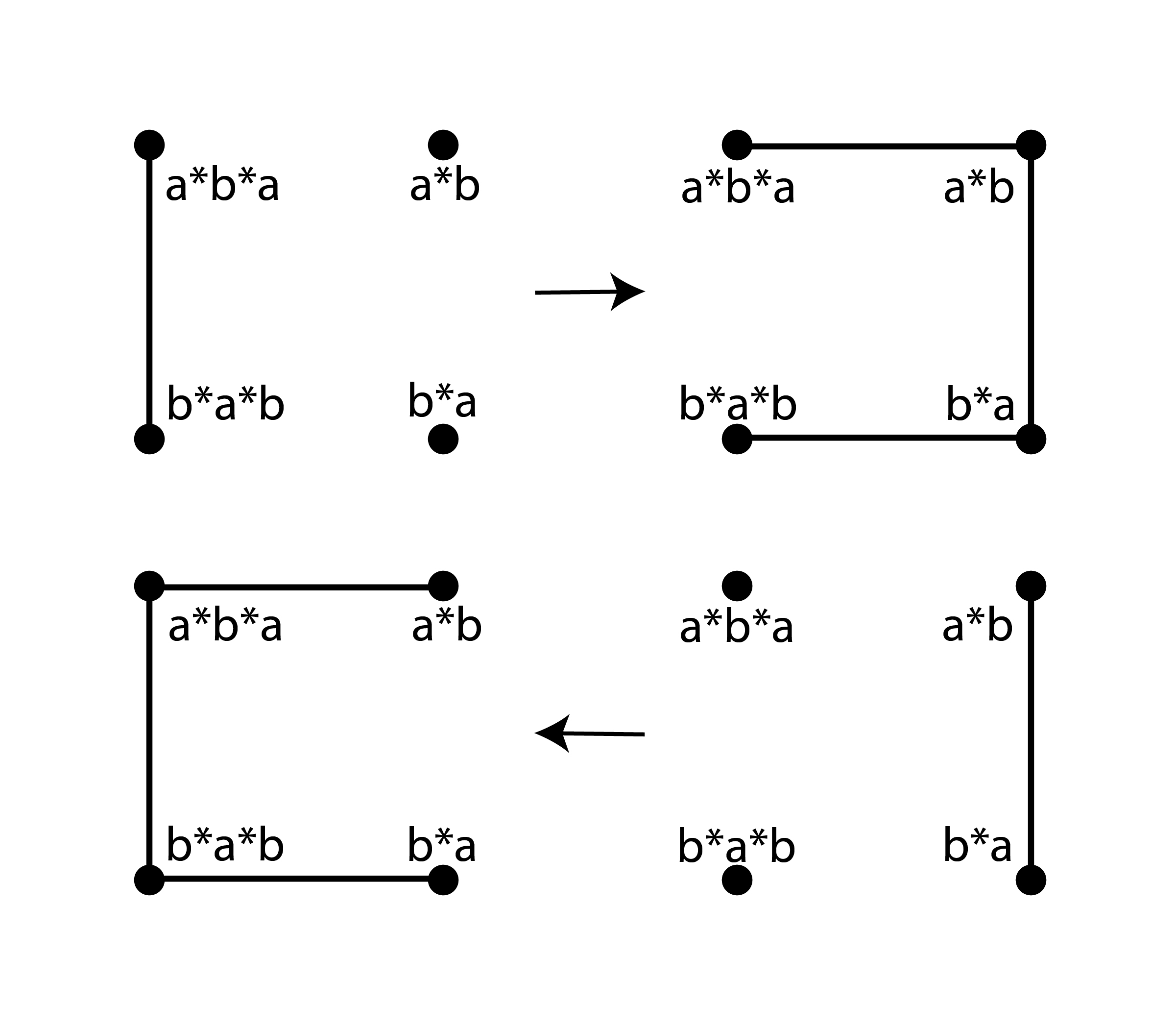}
	\caption{The geometric interpretation of Proposition \ref{prop 2.2}.}
	\label{equivalence_braid_relation_commutativity_geometrically}
\end{figure}

\begin{proposition}\label{prop 2.2} \
	\begin{enumerate}
		\item { Let $(X,*)$ be a semi-group satisfying $a*b*b*c = a*b*c$ for $a,b,c \in X$. Then, $$H_0(X) = \frac{\mathbb{Z}X}{\sim},$$ where $\sim$ is generated in the following way. For $a,b \in X$, $a \sim b$ if there exist $x,y \in X$ such that $a=x*y*x$ and $b=y*x*y.$}
		\item {Let $(X,*)$ be a proto unital shelf or an idempotent semi-group. Then, $$H_0(X) = \frac{\mathbb{Z}X}{\approx},$$ where $\approx$ is generated in the following way. For $a,b \in X$, $a \approx b$ if there exist $x,y \in X$ such that $a=x*y$ and $b=y*x.$ In particular, if $(X,*)$ is commutative, $H_0(X) = \mathbb{Z}X.$}
	\end{enumerate}
\end{proposition}

\begin{proof} \
	\begin{enumerate}
		\item {It follows immediately from the definition of the first boundary map, namely, $\partial_1(a,b) = a*b*a - b*a*b$, for $a,b \in X.$}
		\item{When $(X,*)$ is a proto unital shelf it follows because $x*y = y*x*y$ and $x*y*x = y*x$ for $x,y \in X$. When $(X,*)$ is an idempotent semi-group, it has to be shown that relations $\sim$ and $\approx$ are equal.
			Let $x*y \approx y*x$. Then, $$(y*x)*y \approx (x*y)*y = x*y \approx y*x = (y*x)*x \approx (x*y)*x.$$
			On the other hand, let $x*y*x \sim y*x*y$. Then, 
			\begin{equation*}
				\begin{split}
				& y*x = y*x*y*x = y*x*y*y*x \sim y*(y*x)*y = y*x*y \\ & \sim x*y*x = x*(x*y)*x \sim x*y*x*x*y = x*y*x*y = x*y.
				\end{split}
			\end{equation*}
}
	\end{enumerate}
\end{proof}

\begin{table}[h]
	\caption{An idempotent semi-group (on the left) and a semi-group neither satisfying the axioms of a proto unital shelf nor the idempotence axiom (on the right).}
	\label{IS and non IS}
	\begin{minipage}{0.5\textwidth}
		\centering
		\begin{tabular}{c | c c c c}
			$*$&0&1&2&3 \\
			\hline
			0&0&0&0&0\\
			1&0&1&1&3\\
			2&0&1&2&3\\
			3&0&1&1&3\\	
		\end{tabular}
	\end{minipage}
	\hfillx
	\begin{minipage}{0.5\textwidth}
		\centering
		\begin{tabular}{c|c c c c}
			$*$&0&1&2&3\\
			\hline
			0&0&0&2&2\\
			1&0&0&2&2\\
			2&0&0&2&2\\
			3&0&1&2&3\\
		\end{tabular}
	\end{minipage}
\end{table}

The zero$^{th}$ lbo homology group of semi-groups satisfying $a*b*c=a*b*b*c$ can't be computed in general just by studying commutativity in the multiplication table as the equivalence of the braid group relation and the commutativity axiom does not hold in general. For example, consider the semi-group $X$ shown on the right in Table \ref{IS and non IS}. $H_0(X) = \mathbb{Z}^3$, which would not be the result if the above proposition is used to compute $H_0(X)$. The example below illustrates how the above proposition is used to compute the $0^{th}$ lbo homology group of idempotent semi-groups and proto unital shelves.

\begin{example}
	Consider the idempotent semi-group of four elements shown in Table \ref{IS and non IS}. There are two places in the multiplication table where the commutativity axiom is not satisfied. $$1*3 \neq 3*1,2*3 \neq 3*2.$$
	Therefore, $X = \{0\} \cup \{1,3\} \cup \{2\}$ is the partition from Proposition \ref{prop 2.2}. The number of equivalence classes is $3$. Hence, $H_0(X) = \mathbb{Z}^3.$
\end{example}

Let $(X,*)$ be a finite proto unital shelf or a finite idempotent semi-group. Then $H_0(X)$ is the zero homology group of the $1$-skeleton in the geometric realization. The zero homology group counts the number of components of the geometric structure. In the $1$-skeleton, vertices are joined by edges (bringing down the number of components) when $a*b*a \neq b*a*b$. In the case of proto unital shelves or idempotent semi-groups this equivalently happens when $a*b \neq b*a$. In particular, the edge $(a,b)$ and $(b,a)$ joins the components corresponding to $a*b$ and $b*a$. Translating the language of equivalent classes to components, this basically means that two vertices $x$ and $y$ are in the same component of the $1$-skeleton if for two vertices $a,b$, the image of the edge $(a,b)$ is not zero, $x = a*b$ and $y = b*a$. In particular, when $(X,*)$ is commutative, the number of components is equal to the number of elements in $X$ and therefore the zero homology group of the $1$-skeleton is $Z^{|X|}$.

\section{Temperley-Lieb algebra, Jones' monoids and lbo homology}

The notion of a Temperley-Lieb algebra was introduced in 1971 by Harold N. V. Temperley and Elliott H. Lieb \cite{TL}. A visually appealing diagrammatic definition of the algebra was given by Louis H. Kauffman \cite{Kau}. Briefly, the idea is as follows \cite{Abr,PT}. 

\begin{figure}[ht]
	%\blankbox{.6\columnwidth}{5pc}
	\includegraphics[scale=0.2]{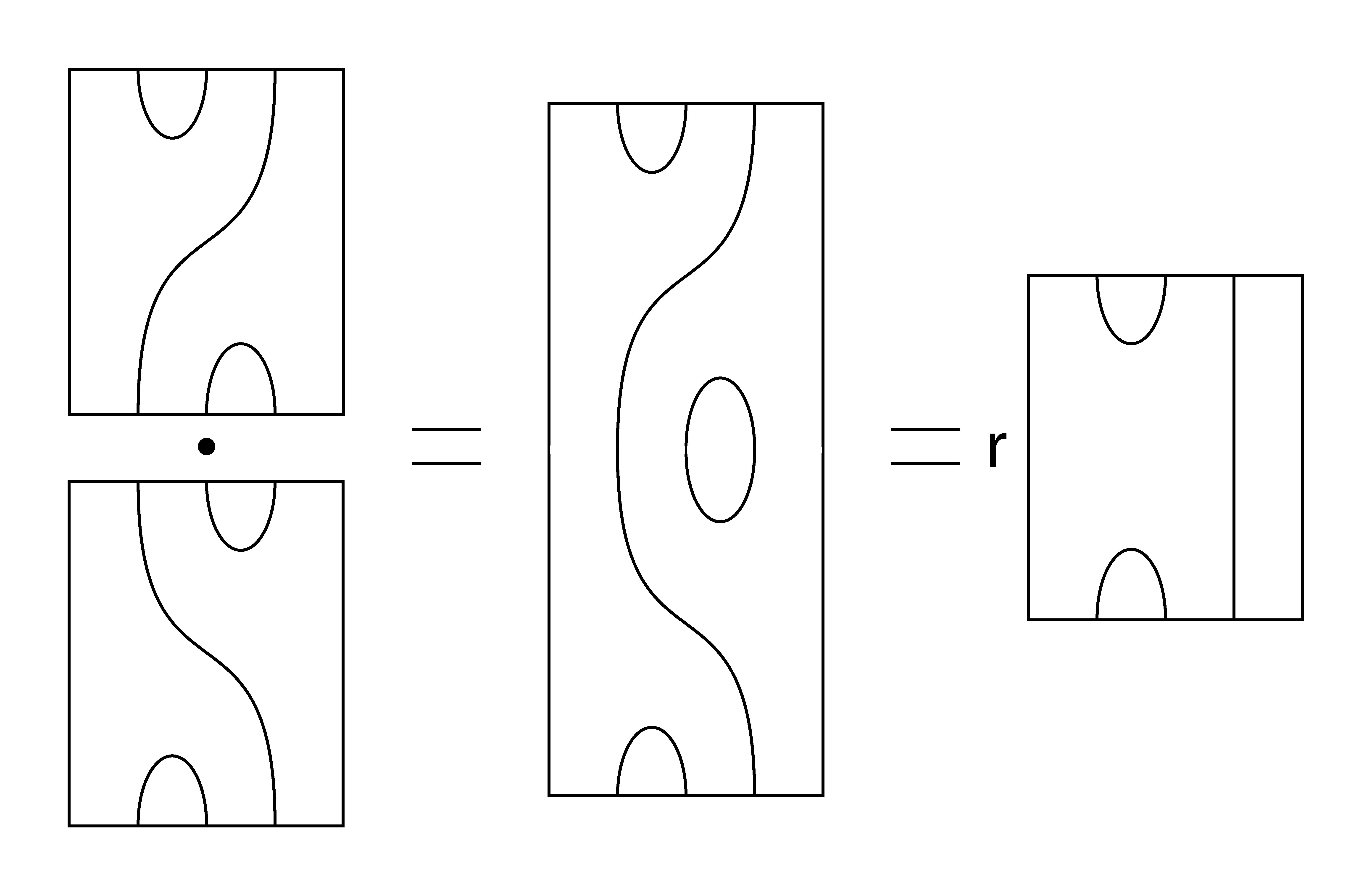}
	\caption{Demonstration of multiplication in Temperley-Lieb algebra.}
	\label{demonstration of multiplication in Temperley-Lieb algebra}
\end{figure}

Consider rectangles with $n$ marked points on the upper side and another $n$ marked points on the lower side. The marked points are then paired with strings so that the strings do not intersect. These diagrams are identified up to planar isotopy and are called {\bf Kauffman diagrams}. The number of Kauffman diagrams for $2n$ marked points is equal to the $n^{th}$ Catalan number. Multiplication between two such diagrams $A$ and $B$ is done by identifying the lower side of the rectangle corresponding to $A$ with the upper side of the rectangle corresponding to $B$. In other words, the diagrams are concatenated or stacked. Closed loops are eliminated by multiplying by $r$ where $R$ is a commutative ring and $r \in R$, a fixed element. An example of the operation and elimination of the closed loop is shown in Figure \ref{demonstration of multiplication in Temperley-Lieb algebra}. For $2n$ points, the $n$-strand {\bf Temperley-Lieb algebra} denoted by $TL_n(r)$ is defined as an $R$-linear algebra spanned by the Kauffman diagrams. 

Equivalently, Temperley-Lieb algebras can be defined using generators and relations. Moreover, the set of Kauffman diagrams for $n$ pairs of points can be given a monoid structure by treating $r$ as a generator along with the standard generating set. Considering two Kauffman diagrams equivalent up to elimination of trivial components the algebraic structure obtained is called the {\bf Jones' monoid} \cite{DE,DEEFHHL}. The Jones' monoid obtained from $TL_n(r)$ is denoted by $J_n$. 

The primary reason for looking at Temperley-Lieb algebras in connection to lbo homology is the Jones' monoids. It is not difficult to observe that many of the elements in these monoids are idempotent with the concatenation operation. Further, the operation is also associative. The question one should ask at this time is whether or not these idempotent elements can be identified and whether or not they form a closed sub-structure inside the Jones' monoids. Well, it turns out to be not very obvious.

All the five elements in $J_3$ are idempotent elements. However, two of the fourteen elements in $J_4$ are not idempotent. Moreover, the twelve idempotent elements are not closed under the concatenation operation. Figure \ref{non-idempotent elements in $J_4$} shows the non idempotent elements in $J_4$. Notice that they are mirror images of each other.

\begin{figure}[ht]
%\blankbox{.6\columnwidth}{5pc}
\includegraphics[scale=0.3]{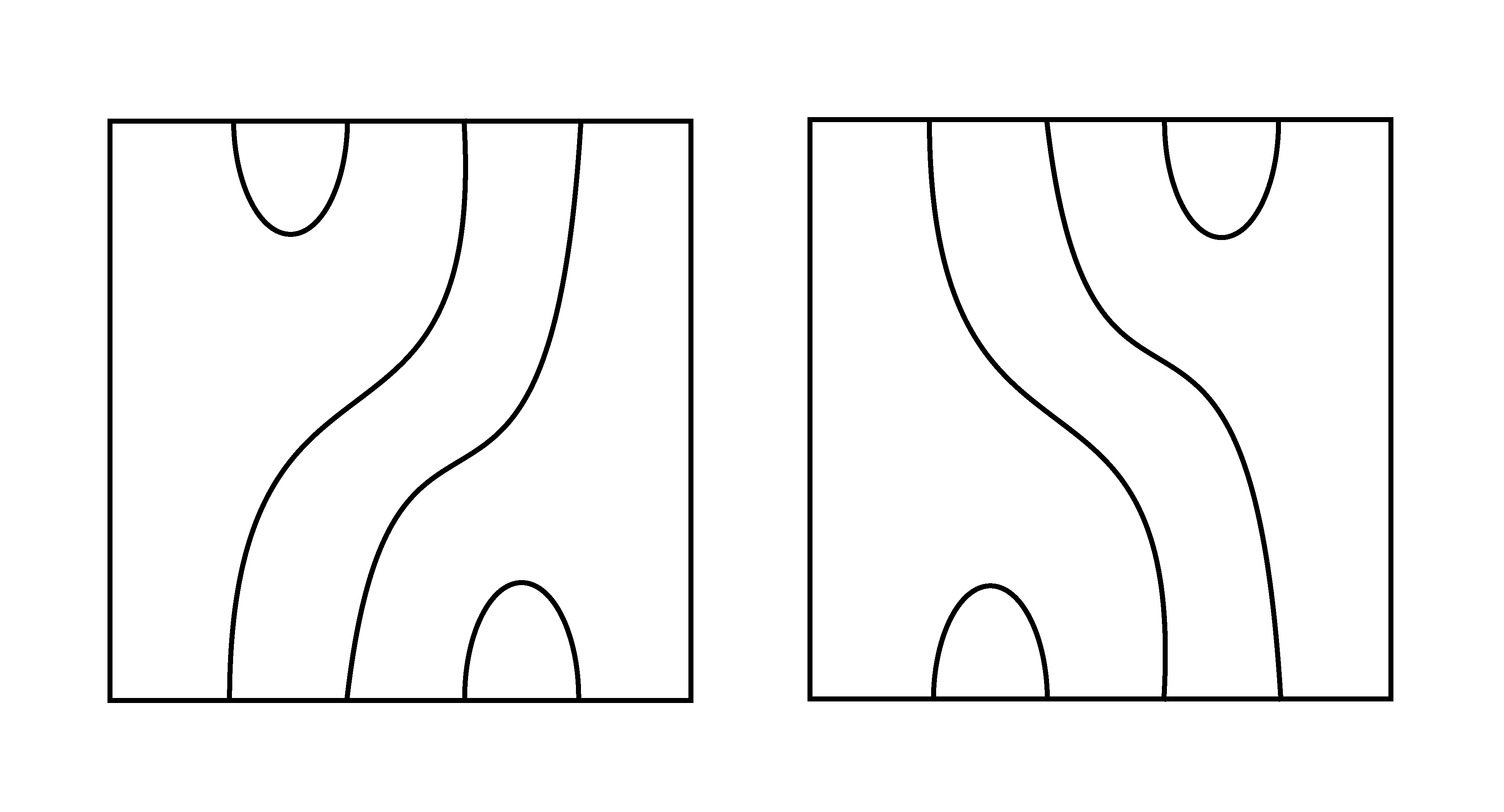}
\caption{Non-idempotent elements in $J_4$.}
\label{non-idempotent elements in $J_4$}
\end{figure}

I. Dolinka et al. characterized the idempotent elements in the Jones' monoid \cite{DEEFHHL}. Using the notion of an interface graph,  they proved the following proposition.

\begin{proposition}[I. Dolinka et al.]
		An element in the Jones' monoid is idempotent iff its interface graph has no odd path.
\end{proposition}   

The following table displays the number of idempotents in the Jones' monoids $J_n$ with growing $n$. The remark after the table provides the pathway for constructing families of idempotent sub-monoids of $J_n$.

\begin{table}[h]\label{number of idempotents}
	\caption{Idempotents in $J_n$.}
	\centering
	\begin{tabular}{| c | c | c | c | c | c | c | c | c |}
		\hline
		1&2&3&4&5&6&7&8&9 \\
		\hline
		1&2&5&12&36&96&311&886&3000 \\	
		\hline
	\end{tabular}
\end{table}

\begin{remark}
	The number of partitions of the integer $n$ whose largest part is $k$ is equal to the number of partitions of $n$ with $k$ parts. If this value for given $n$ is denoted by $n_k$ then $n_k = (n-k)_k + (n-1)_{k-1}.$ The value of $k$ which is of interest in the current context is $3$. Then the recurrence relation becomes $n_3 =  (n-3)_3 + (n-1)_2.$ The first few terms in this sequence starting for $n=1$ are $1,2,3,4,5,7,8,10,12,14...$.
\end{remark}

\begin{definition}
	Let $j \in J_n$. $P = a_1 + a_2 + ... + a_k$ for some positive $k$ is said to a {\bf partition} of $j$ if $j$ can be divided with $(k-1)$ vertical line segments with the divided parts having $a_1,a_2,...,a_k$ strings. 
\end{definition}

Figure \ref{partitioning elements in $J_n$} demonstrates the idea of partitioning elements in Jones' monoids. In the figure, $ n = 10 $ and $ P = 2 + 1 + 2 + 5.$ Notice that the identity element in $J_n$ admits every partition of $P$.

\begin{figure}[ht]
	%\blankbox{.6\columnwidth}{5pc}
	\includegraphics[scale=0.4]{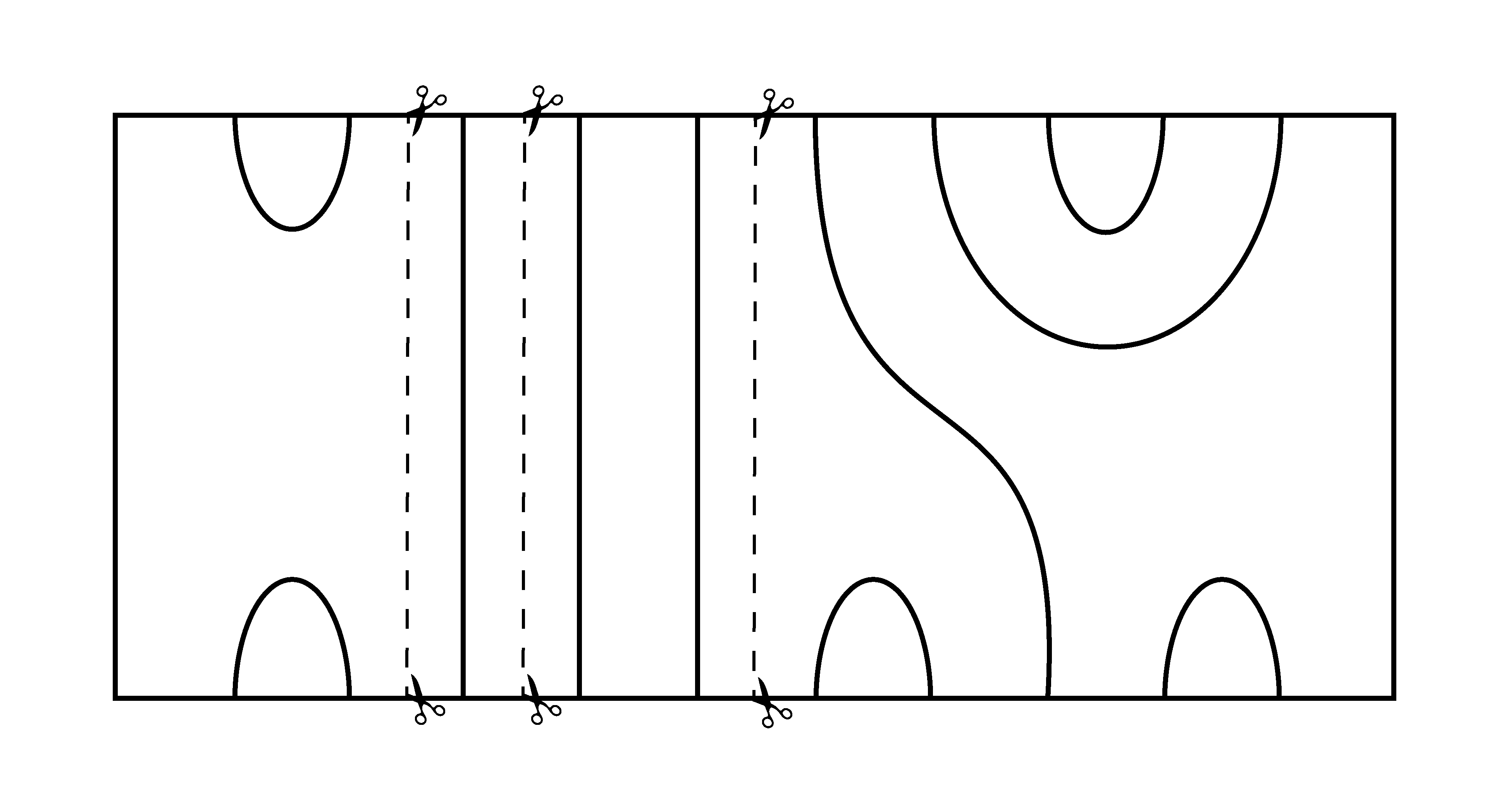}
	\caption{Partitioning elements in Jones' monoids.}
	\label{partitioning elements in $J_n$}
\end{figure}

\begin{proposition}
	Let $P$ be a partition of the natural number $n$ with each part having value at most $3$. Denote by $JSMP_n$\footnote{Jones' sub-monoid corresponding to partition $P$ in $J_n$ is abbreviated as $JSMP_n$.} the set of elements in $J_n$ for which $P$ is a partition. Then $JSMP_n$ is an idempotent monoid.
\end{proposition}

\begin{proof}
	As each part of the partition $P$ has value at most $3$, individually each of these are elements of $J_1$, $J_2$, or $J_3$ and therefore idempotent. Globally, that is for the entire element idempotence follows as none of the parts interact with each other.
\end{proof}

Note that two Jones' sub-monoids in $J_n$ are isomorphic if and only if they correspond to the same partition. Moreover, the trivial Jones' sub-monoid is present in every $J_n$ as every $J_n$ admits the partition $1+1+...+1$($n$-times). As is obvious by now, lbo homology can be computed for Jones' sub-monoids. Since, Jones' sub-monoids are idempotent semi-groups, they are not separately enumerated in this note. The reason behind eliminating trivial components in Kauffman diagrams earlier is that all the elements when composed with themselves do not generate the same number of trivial components and by definition the face maps of lbo homology do not allow this.

\section{Odds and ends}

There is much more to discover about lbo homology. Some ideas and experimental data are summarized in this section.

\subsection{Preliminary computational data}

This part of the note is dedicated towards preliminary computational data of lbo homology. To avoid lengthening the note unnecessarily, in this section the algebraic structures will be presented in programming notation. For example, the associative shelf in Table \ref{IS and AS} is presented by $\{\{0,0,0,0\},\{0,0,1,1\},\{0,0,2,2\},\{0,0,2,3\}\}$. The rows with a \Coffeecup (coffee cup) symbol on the left in Table \ref{computational data for associative shelves} are also idempotent semi-groups and the rows with a \Coffeecup (coffee cup) symbol on the left in Table \ref{computational data for idempotent semi-groups} are also associative shelves. Table \ref{computational data for abcabbc semi-groups} shows computations for semi-groups satisfying the axiom $a*b*b*c = a*b*c$ but not idempotence or right self-distributivity.

\begin{table}[h]
	\caption{Lbo homology of associative shelves up to order $3$.}
	\label{computational data for associative shelves}
	\centering
	\begin{tabular}{| c | c | c | c | c | c |}
		\hline
		& Associative shelf: $X$ & $H_0(X)$ & $H_1(X)$ & $H_2(X) $ & $ H_3(X)$ \\
		\hline
		& $\{ \{ 0,0 \},\{ 0,0 \} \}$ & $\mathbb{Z}^2$ & $\mathbb{Z}^3$ & $\mathbb{Z}^4$ & $\mathbb{Z}^7$ \\	
		\hline
		\Coffeecup & $\{ \{ 0,0 \},\{ 0,1 \} \}$ & $\mathbb{Z}^2$ & $\mathbb{Z}$ & 0 & $\mathbb{Z}$ \\
		\hline
		\Coffeecup & $\{ \{ 0,0 \},\{ 1,1 \} \}$ & $\mathbb{Z}$ & 0 & 0 & 0 \\	
		\hline
		\Coffeecup & $\{ \{ 0,1 \},\{ 0,1 \} \}$ & $\mathbb{Z}$ & 0 & 0 & 0 \\
		\hline
		& $\{ \{ 0,0,0 \},\{ 0,0,0 \},\{ 0,0,0 \} \}$ & $\mathbb{Z}^3$ & $\mathbb{Z}^8$ & $\mathbb{Z}^{20}$ & $\mathbb{Z}^{57}$ \\
		\hline
		& $\{ \{ 0,0,0 \},\{ 0,0,0 \},\{ 0,0,1 \} \}$ & $\mathbb{Z}^3$ & $\mathbb{Z}^6$ & $\mathbb{Z}^9$ & $\mathbb{Z}^{19}$ \\
		\hline
		& $\{ \{ 0,0,0 \},\{ 0,0,0 \},\{ 0,0,2 \} \}$ & $\mathbb{Z}^3$ & $\mathbb{Z}^6$ & $\mathbb{Z}^7$ & $\mathbb{Z}^{21}$ \\
		\hline
		& $\{ \{ 0,0,0 \},\{ 0,0,0 \},\{ 2,2,2 \} \}$ & $\mathbb{Z}^2$ & $\mathbb{Z}^5$ & $\mathbb{Z}^7$ & $\mathbb{Z}^{17}$ \\
		\hline
		& $\{ \{ 0,0,0 \},\{ 0,0,1 \},\{ 0,0,2 \} \}$ & $\mathbb{Z}^3$ & $\mathbb{Z}^3$ & $\mathbb{Z}$ & $\mathbb{Z}^{5}$ \\
		\hline
		\Coffeecup & $\{ \{ 0,0,0 \},\{ 0,1,0 \},\{ 0,0,2 \} \}$ & $\mathbb{Z}^3$ & $\mathbb{Z}^4$ & 0 & $\mathbb{Z}^6$ \\
		\hline
		\Coffeecup & $\{ \{ 0,0,0 \},\{ 0,1,0 \},\{ 2,2,2 \} \}$ & $\mathbb{Z}^2$ & $\mathbb{Z}$ & 0 & $\mathbb{Z}$ \\
		\hline
		& $\{ \{ 0,0,0 \},\{ 0,1,1 \},\{ 0,1,1 \} \}$ & $\mathbb{Z}^3$ & $\mathbb{Z}^6$ & $\mathbb{Z}^7$ & $\mathbb{Z}^{21}$ \\
		\hline
		\Coffeecup & $\{ \{ 0,0,0 \},\{ 0,1,1 \},\{ 0,1,2 \} \}$ & $\mathbb{Z}^3$ & $\mathbb{Z}^3$ & 0 & $\mathbb{Z}^7$ \\
		\hline
		\Coffeecup & $\{ \{ 0,0,0 \},\{ 0,1,1 \},\{ 0,2,2 \} \}$ & $\mathbb{Z}^2$ & $\mathbb{Z}^2$ & 0 & $\mathbb{Z}^4$ \\
		\hline
		\Coffeecup & $\{ \{ 0,0,0 \},\{ 0,1,2 \},\{ 0,1,2 \} \}$ & $\mathbb{Z}^2$ & $\mathbb{Z}^2$ & 0 & $\mathbb{Z}^4$ \\
		\hline
		\Coffeecup & $\{ \{ 0,0,0 \},\{ 1,1,1 \},\{ 2,2,2 \} \}$ & $\mathbb{Z}$ & 0 & 0 & 0 \\
		\hline
		& $\{ \{ 0,0,2 \},\{ 0,0,2 \},\{ 0,0,2 \} \}$ & $\mathbb{Z}^2$ & $\mathbb{Z}^5$ & $\mathbb{Z}^7$ & $\mathbb{Z}^{17}$ \\
		\hline
		\Coffeecup & $\{ \{ 0,0,2 \},\{ 0,1,2 \},\{ 0,0,2 \} \}$ & $\mathbb{Z}^2$ & $\mathbb{Z}$ & 0 & $\mathbb{Z}$ \\
		\hline
		\Coffeecup & $\{ \{ 0,0,2 \},\{ 0,1,2 \},\{ 0,2,2 \} \}$ & $\mathbb{Z}^2$ & $\mathbb{Z}^2$ & 0 & $\mathbb{Z}^4$ \\
		\hline
		\Coffeecup & $\{ \{ 0,1,2 \},\{ 0,1,2 \},\{ 0,1,2 \} \}$ & $\mathbb{Z}$ & 0 & 0 & 0 \\
		\hline
	\end{tabular}
\end{table}

\begin{table}[h]
	\caption{Lbo homology of some idempotent semi-groups up to order $4$.}
	\label{computational data for idempotent semi-groups}
	\centering
	{\small
	\begin{tabular}{ | c | c | c | c | c | c |}
		\hline
		& Idempotent semi-group: $X$ & $H_0(X)$ & $H_1(X)$ & $H_2(X) $ & $ H_3(X)$ \\
		\hline
		& $\{ \{ 0,0,0 \},\{ 0,1,2 \},\{ 2,2,2 \} \}$ & $\mathbb{Z}^2$ & $\mathbb{Z}^2$ & 0 & $\mathbb{Z}^4$ \\
		\hline
		\Coffeecup & $\{ \{ 0,0,0,0 \},\{ 0,1,0,0 \},\{ 0,0,2,0 \},\{0,0,0,3\} \}$ & $\mathbb{Z}^4$ & $\mathbb{Z}^9$ & $\mathbb{Z}^6$ & $\mathbb{Z}^{27}$ \\
		\hline
		\Coffeecup & $\{ \{ 0,0,0,0 \},\{ 0,1,0,0 \},\{ 0,0,2,0 \},\{3,3,3,3\} \}$ & $\mathbb{Z}^3$ & $\mathbb{Z}^4$ & 0 & $\mathbb{Z}^6$ \\
		\hline
		\Coffeecup & $\{ \{ 0,0,0,0 \},\{ 1,1,1,1 \},\{ 2,2,2,2 \},\{3,3,3,3\} \}$ & $\mathbb{Z}$ & 0 & 0 & 0 \\
		\hline
		& $\{ \{ 0,0,0,0 \},\{ 0,1,1,1 \},\{ 0,1,2,3 \},\{0,3,3,3\} \}$ & $\mathbb{Z}^3$ & $\mathbb{Z}^5$ & 0 & $\mathbb{Z}^{21}$ \\
		\hline
		& $\{ \{ 0,0,0,0 \},\{ 0,1,2,3 \},\{ 0,2,2,3 \},\{3,3,3,3\} \}$ & $\mathbb{Z}^3$ & $\mathbb{Z}^5$ & 0 & $\mathbb{Z}^{21}$ \\
		\hline
		& $\{ \{ 0,0,0,0 \},\{ 0,1,1,3 \},\{ 0,2,2,3 \},\{3,3,3,3\} \}$ & $\mathbb{Z}^2$ & $\mathbb{Z}^4$ & 0 & $\mathbb{Z}^{16}$ \\
		\hline
		& $\{ \{ 0,0,0,0 \},\{ 0,1,2,3 \},\{ 0,1,2,3 \},\{3,3,3,3\} \}$ & $\mathbb{Z}^2$ & $\mathbb{Z}^4$ & 0 & $\mathbb{Z}^{16}$ \\
		\hline
		& $\{ \{ 0,0,0,0 \},\{ 0,1,2,3 \},\{ 2,2,2,2 \},\{3,3,3,3\} \}$ & $\mathbb{Z}^2$ & $\mathbb{Z}^3$ & 0 & $\mathbb{Z}^9$ \\
		\hline
		\Coffeecup & $\{ \{ 0,0,0,0 \},\{ 0,1,0,1 \},\{ 0,0,2,2 \},\{0,1,2,3\} \}$ & $\mathbb{Z}^4$ & $\mathbb{Z}^5$ & 0 & $\mathbb{Z}^{15}$ \\
		\hline
		\Coffeecup & $\{ \{ 0,1,2,3 \},\{ 0,1,2,3 \},\{ 0,1,2,3 \},\{0,1,2,3\} \}$ & $\mathbb{Z}$ & 0 & 0 & 0 \\
		\hline
	\end{tabular}}
\end{table}

\begin{table}[h]
	\caption{Lbo homology of semi-groups satisfying $a*b*c = a*b*b*c$ up to order $4$.}
	\label{computational data for abcabbc semi-groups}
	\centering
	{\small
		\begin{tabular}{ | c | c | c | c | c |}
			\hline
			Algebraic structure: $X$ & $H_0(X)$ & $H_1(X)$ & $H_2(X) $ & $ H_3(X)$ \\
			\hline
			$\{ \{ 0,0,0 \},\{ 0,0,0 \},\{0,1,2 \} \}$ & $\mathbb{Z}^3$ & $\mathbb{Z}^3$ & $\mathbb{Z}$ & $\mathbb{Z}^5$ \\
			\hline
			$\{ \{ 0,0,0,0 \},\{ 0,0,0,0 \},\{ 0,0,0,0 \},\{0,0,2,3\} \}$ & $\mathbb{Z}^4$ & $\mathbb{Z}^{10}$ & $\mathbb{Z}^{17}$ & $\mathbb{Z}^{58}$ \\
			\hline
			$\{ \{ 0,0,0,0 \},\{ 0,0,0,0 \},\{ 0,0,0,0 \},\{0,1,1,3\} \}$ & $\mathbb{Z}^4$ & $\mathbb{Z}^{10}$ & $\mathbb{Z}^{17}$ & $\mathbb{Z}^{58}$ \\
			\hline
			$\{ \{ 0,0,0,0 \},\{ 0,0,0,0 \},\{ 0,0,0,0 \},\{0,1,2,3\} \}$ & $\mathbb{Z}^4$ & $\mathbb{Z}^{7}$ & $\mathbb{Z}^{8}$ & $\mathbb{Z}^{25}$ \\
			\hline
			$\{ \{ 0,0,0,0 \},\{ 0,0,0,0 \},\{ 0,0,0,2 \},\{0,1,0,3\} \}$ & $\mathbb{Z}^4$ & $\mathbb{Z}^{11}$ & $\mathbb{Z}^{3}$ & $\mathbb{Z}^{11}$ \\
			\hline
			$\{ \{ 0,0,0,0 \},\{ 0,0,0,0 \},\{ 0,0,2,0 \},\{0,1,0,3\} \}$ & $\mathbb{Z}^4$ & $\mathbb{Z}^{8}$ & $\mathbb{Z}^{7}$ & $\mathbb{Z}^{26}$ \\
			\hline
			$\{ \{ 0,0,0,0 \},\{ 0,1,1,1 \},\{ 0,1,1,1 \},\{0,1,2,3\} \}$ &  $\mathbb{Z}^4$ & $\mathbb{Z}^{6}$ & $\mathbb{Z}$ & $\mathbb{Z}^{22}$ \\
			\hline
			$\{ \{ 0,0,0,0 \},\{ 0,1,1,3 \},\{ 0,1,1,3 \},\{3,3,3,3\} \}$ &  $\mathbb{Z}^3$ & $\mathbb{Z}^{9}$ & $\mathbb{Z}^{10}$ & $\mathbb{Z}^{47}$ \\
			\hline
			$\{ \{ 0,0,0,3 \},\{ 0,0,0,3 \},\{ 0,1,2,3 \},\{0,0,0,3\} \}$ &  $\mathbb{Z}^3$ & $\mathbb{Z}^{3}$ & $\mathbb{Z}$ & $\mathbb{Z}^{5}$ \\
			\hline
			$\{ \{ 0,0,0,3 \},\{ 0,0,0,3 \},\{ 0,1,2,3 \},\{0,0,3,3\} \}$ & $\mathbb{Z}^3$ & $\mathbb{Z}^{5}$ & $\mathbb{Z}$ & $\mathbb{Z}^{17}$ \\
			\hline
			$\{ \{ 0,0,2,2 \},\{ 0,0,2,2 \},\{ 0,0,2,2 \},\{0,1,2,3\} \}$ & $\mathbb{Z}^3$ & $\mathbb{Z}^{3}$ & $\mathbb{Z}$ & $\mathbb{Z}^{5}$ \\
			\hline
	\end{tabular}}
\end{table}

\subsection{A comparison between the homology theories}

This subsection is dedicated towards comparing the different homology theories for associative shelves. In particular, group homology, Hochschild homology, lbo homology, one term homology and rack homology are discussed. As the initial approach to lbo homology was from the side of self-distributive algebraic structures, group homology and Hochschild homology are not studied in detail. However, one can explore more in this direction. Unless otherwise stated, the notation for lbo homology remains unchanged. A brief description of the other four homology theories are as follows \cite{Prz1}.

\begin{definition}
	Let $(X, *)$ be a semi-group. For $ n \geq 0$, let $C_n = \mathbb{Z}X^{n+1}$ and $\partial_n : C_n \longrightarrow C_{n-1}$ given by:
	\begin{equation*}
	\begin{split}
	\partial_n(x_0,x_1,..., x_n) & = (x_1,x_2,...,x_n) + \sum_{i=1}^{n-1}(-1)^i(x_0,x_1,...,x_i * x_{i+1},x_{i+2},...,x_n) \\
	&+ (-1)^n(x_0,x_1,...,x_{n-1}).
	\end{split}
	 \end{equation*}
	 Then, $(C_n,\partial_n)$ is a chain complex and the {\bf group homology} groups denoted by $H_*^G(X)$ are defined in the usual way.
\end{definition}

\begin{definition}
	Let $(X, *)$ be a semi-group. For $ n \geq 0$, let $C_n = \mathbb{Z}X^{n+1}$ and $\partial_n : C_n \longrightarrow C_{n-1}$ given by:{\small
	\begin{equation*}
		\partial_n(x_0,x_1,..., x_n)  = \sum_{i=0}^{n-1}(-1)^i(x_0,x_1,...,x_i * x_{i+1},x_{i+2},...,x_n) + (-1)^n(x_n*x_0,x_1,...,x_{n-1}).
	\end{equation*}}
	Then, $(C_n,\partial_n)$ is a chain complex and the {\bf Hochschild homology} groups denoted by $H_*^H(X)$ are defined in the usual way.
\end{definition}

\begin{definition}
	Let $(X, *)$ be a shelf. For $ n \geq 0$, let $C_n = \mathbb{Z}X^{n+1}$ and $\partial_n : C_n \longrightarrow C_{n-1}$ given by: {\small
	\begin{equation*}
		\partial_n(x_0,x_1,..., x_n) = (x_1,x_2,...,x_n) + \sum_{i=1}^{n}(-1)^i(x_0*x_i,x_1*x_i,...,x_{i-1} * x_i,x_{i+1},x_{i+2},...,x_n).
	\end{equation*}}
	Then, $(C_n,\partial_n)$ is a chain complex and the {\bf one term homology} groups denoted by $H_*^O(X)$ are defined in the usual way.
\end{definition}

\begin{definition}
	Let $(X, *)$ be a shelf. For $ n \geq 0$, let $C_n = \mathbb{Z}X^{n+1}$ and $\partial_n : C_n \longrightarrow C_{n-1}$ given by:
		\begin{equation*}
		\begin{split}
		\partial_n(x_0,x_1,..., x_n) = \sum_{i=1}^{n}(-1)^i\{ &(x_0,x_1,...,x_{i-1},x_{i+1},x_{i+2},...x_n) \\
		- &(x_0*x_i,x_1*x_i,...,x_{i-1} * x_i,x_{i+1},x_{i+2},...,x_n)\}.
		\end{split}
		\end{equation*}
	Then, $(C_n,\partial_n)$ is a chain complex and the {\bf rack homology} groups denoted by $H_*^R(X)$ are defined in the usual way.
\end{definition}

The one term and rack homology groups of associative shelves were studied in \cite{CMP}. In particular, it was proven that for associative shelves with a right fixed element the rack homology groups are $\mathbb{Z}$ in all dimensions. Moreover, it was observed that proto unital shelves always have right zero. Table \ref{computational data for associative shelves} indicates that lbo homology groups of proto unital shelves are not constant. Unital shelves have trivial one term homology groups in all positive dimensions. In particular, one term homology groups of shelves with either a left zero or a right unit are trivial in all positive dimensions \cite{Prz1}. Again, from Table \ref{computational data for associative shelves} one may infer that lbo homology of such shelves is more interesting.

\subsection{Open questions}

With the very limited computational data that is there, one might suspect some patterns in lbo homology. They are as follows.

\begin{remark}
	Let $(X,*)$ be an idempotent semi-group of size $n$ for $n>3$. Tables 5 and 6 suggest that $H_2(X) = 0$ for all but the trivial idempotent semi-group for each $n$. It would be interesting to find the reason behind this!
\end{remark}

\begin{conjecture}
	There is no torsion in lbo homology.
\end{conjecture}	
 One possible way to attack the above conjecture is by studying the homotopy type of the CW complex formed by the geometric realization arising from lbo homology. 

\begin{remark}
	There are not many geometric interpretations of the idempotence axiom. However, with such an interpretation, lbo homology may turn out to be a very useful invariant for those geometric structures assuming it would not be difficult to convert the associativity axiom in a similar manner. In connection to knot theory it might be useful to consider coloring knotted trivalent graphs.
\end{remark}

\section{Acknowledgements}

The author would like to thank Louis H. Kauffman, J\'ozef H. Przytycki, and Masahico Saito for their useful comments and suggestions.

\end{document}